\documentclass{amsart}

\usepackage{amsmath}
\usepackage{amscd}
\usepackage{amsthm}
\usepackage{amsfonts}
\usepackage{amssymb}
\usepackage{mathrsfs}
\usepackage{mathabx}
\usepackage{stmaryrd}
\usepackage{bbm}
\usepackage{cancel}
\usepackage[enableskew]{youngtab}
\usepackage[all]{xy}

\newtheorem{proposition}{Proposition}
\newtheorem{theorem}[proposition]{Theorem}
\newtheorem{lemma}[proposition]{Lemma}
\newtheorem{conjecture}[proposition]{Conjecture}

\theoremstyle{remark}

\newtheorem{example}[proposition]{Example}
\newtheorem{remark}[proposition]{Remark}

\theoremstyle{definition}
\newtheorem{definition}[proposition]{Definition}

\newcommand{\rk}{{\rm rk}}
\newcommand{\z}{\mathbbm Z}
\renewcommand{\r}{\mathbbm R}
\renewcommand{\c}{\mathbbm C}
\newcommand{\n}{\mathbbm N}

\newcommand{\dirac}{\cancel{\partial}}

\newcommand{\yc}{\Yvcentermath1}

\begin{document}

\title[The decomposition of the spinor bundle of Grassmann manifolds]{The decomposition of the spinor bundle of Grassmann manifolds}
\author{Frank Klinker}
\date{\today}
\address{Department of Mathematics, University of Dortmund, D--44221 Dortmund} 
\address{frank.klinker@mathematik.uni-dortmund.de}
\begin{abstract} 
The decomposition of the spinor bundle of the spin Grassmann manifolds  $G_{m,n}=SO(m+n)/SO(m)\times SO(n)$ into irreducible representations of $\mathfrak{so}(m)\oplus\mathfrak{so}(n)$ is presented. A universal construction is developed and  the general statement is proven for $G_{2k+1,3}$, $G_{2k,4}$, and $G_{2k+1,5}$ for all $k$.
The decomposition is used to discuss properties of the spectrum and the eigenspaces of the Dirac operator.
\end{abstract}
\maketitle

\section{Introduction}

The discussion of the spectrum of differential operators on spin symmetric and spin homogeneous spaces has been part of the literature for many years (see for example, \cite{CampoPedon},\cite{Partha},\cite{Seeger1},\cite{Strese1} or \cite{Milhorat1}). This topic brings together different aspects of geometry and representation theory, such as existence of spin structures on homogeneous spaces  (e.g., \cite{CahenGutt} or \cite{LyakMudrov}) and  branching rules for representations  (e.g., \cite{VasilLyak}, \cite{MacdonaldBook}, \cite{King1982}, \cite{Milhorat3}, \cite{HoweTanWillen}, \cite{KoikeTerada1}, \cite{Koike1}, \cite{KoikeTerada4}, \cite{CastilhoAlcaras}). In particular the work of Parthasarathy \cite{Partha} yields an important theoretical tool to describe the spectrum and the eigenspaces of the Dirac operator of a spin symmetric space $G/K$. It may roughly be summarized as follows. 
In the first step, decompose the spinor bundle of $M$ in irreps of $K$. 
In the second step, list all $G$-representation which decomposition with respect to  $K$ admits a summand from the list obtained in the first  step. 
In \cite{Milhorat1} and \cite{Milhorat2} these tools have been noticed to be very powerful  for the discussion of the first eigenvalue of the Dirac operator.
Nevertheless the practical application of the theoretical tools  contains many  difficulties which are of course the reason why most authors, including ourselves, restrict to examples. 
If we consider symmetric spaces, in particular where both parts are of the same rank, the second step has mainly been solved for the classical groups in the literature cited above. For the first step we need branching rules for the isotropy group  of $M$ with respect to the subgroup $K$. Therefore this step is more sophisticated in so far as the rank difference between $K$ and  the isotropy group is big, in general. For example, this difference is $\text{rk}(\mathfrak{so}(4k\ell))-\text{rk}(\mathfrak{so}(2k)\oplus\mathfrak{so}(2\ell))=2k\ell-k-\ell$ for the Grassmannian $G_{2k,2\ell}$. 

In this text we prove a formula for the  decomposition of the spinor bundle of the  spin Grassmannians $G_{m,n}=SO(m+n)/SO(m)\times SO(n)$ for $n\leq 5$. The construction also  yields the decomposition of the spinor bundle of $G_{m,n}$ in the general case. See theorem \ref{decspinorgrassl2} and theorem \ref{decspinorgrassoddpart} as well as  conjecture \ref{decspinorgrass} and conjecture \ref{decspinorgrassodd}. 
This decomposition can be rephrased as branching of the spin representation of $\mathfrak{so}(mn)$ with respect to the subalgebra $\mathfrak{so}(m)\oplus\mathfrak{so}(n)$ for which there are  no general statement proven so far but only partial results for $n=1,2$ (see, e.g., \cite{Strese1}). 
Our proof needs the explicit construction of the weights cf.\ (\ref{matrix}) and (\ref{images}) but also a dimension analysis.  
For the case $(m,n)=(2k,4)$  this dimension analysis is reformulated in lemma \ref{lemma4}. Here as well as in the odd dimensional cases with $n\leq 5$ we need closed expressions for  sums over binomial terms  which we prove in appendix \ref{applemma}. Why and how the techniques provided in the proofs for $n\leq 5$ can be used for the general case and what the practical difficulties are, is explained in section \ref{outlook}. 
Nevertheless for fixed $(m,n)$  one can let any computer algebra system -- for example, {\sc maple} -- do the dimension calculation to tell one that the decomposition results are right in all these cases. 
But this is not the only  evidence  of the correctness  of our general result. In  section \ref{diracsection} we discuss some aspects of the spectrum and the eigenspaces of the Dirac operator on Grassmann manifolds. 
We compare  our results of section \ref{sectionX} with the results of \cite{Milhorat1} and \cite{Milhorat2}. The  perfect match also substantiates  the statement on the decompositions (\ref{dec}) and (\ref{decodd})  in the general case. As a further result we identify the smallest summand of the eigenspace to the first eigenvalue of the Dirac operator, see propositions \ref{repsmall2con} and \ref{smallcontrl2} and conjectures \ref{repsmallgencon} and  \ref{smallcontr}.  Small in this situation means with respect to the ordering which is induced by the ordering of weights. We show that this eigenspace is nondegenerate in the case $G_{4,4}$ (see example \ref{exampleexpansion}).

The text is organized as follows. In section \ref{sec2} we recall the theoretical basis and explain the projection method at a well known simple example before we use this method to decompose the spinor bundle of the Grassmannians in section \ref{sectionX}. In section \ref{diracsection} we turn to the discussion of the Dirac operator and its spectrum and end up with some concluding remarks in section \ref{outlook}.

\section{Decomposing the spinor representation of ${G/K}$}\label{sec2}

\subsection{Symmetric spaces: The Parthasarathy formula}\label{11}

Consider a homogeneous  spaces $M= G/K$, $K\subset G$, $\mathfrak{g}=\mathfrak{k}\oplus\mathfrak{p}$ together with its $G$-invariant Riemannian metric induced by the Killing form $b$.\footnote{For $M$  isotropy irreducible we get an Einstein space with scalar curvature $s=\frac{\dim M}{2}$.} Then 
\begin{gather*}
\mathfrak{p}=\big\{v\in\mathfrak{g}|\,b(v,h)=0\ \forall h\in K\big\}\,,\\
\big[\mathfrak{k},\mathfrak{k}\big]\subset\mathfrak{k}\,,
\quad \big[\mathfrak{k},\mathfrak{p}\big]\subset\mathfrak{p}\,,
\quad \big[\mathfrak{p},\mathfrak{p}\big]\subset\mathfrak{k}\oplus \mathfrak{p}\,.
\end{gather*}
Suppose $M$ is a symmetric space.  Then, in particular, the bracket of $\mathfrak{p}$ with itself closes into $\mathfrak{k}$. 
Suppose $\mathfrak{k}$ and $\mathfrak{g}$ are of the same rank.\footnote{This is true in almost all cases of symmetric spaces up to two series, see, for example, \cite{BookHelgason}.}

Let $\zeta:K\to SO(\mathfrak{p})$ be the isotropy representation of $M$. $\zeta$ induces a representation $\zeta_*$ on Lie algebra level given by
\begin{equation}
\zeta_*:\mathfrak{k}\to\mathfrak{so}(\mathfrak{p}),\quad 
\zeta_*(h)(v)=(ad^\mathfrak{g}_h)\big|_{\mathfrak{p}}(v)={\rm proj}_{\mathfrak{p}}[h,v]\,.
\end{equation}
\begin{remark}
$M=G/K$  admits a $G$-invariant spin structure if and only if $\zeta$ lifts to $\tilde\zeta:K\to Spin(\mathfrak{p})$.
\end{remark}
 Let $\{e_i\}$ be an orthonormal basis of $\mathfrak{p}$. Then $\zeta_*$ and $\tilde\zeta_*:\mathfrak{k}\to\mathfrak{so}(\mathfrak{p})$  are connected via
\begin{equation}
\tilde\zeta_*(h)=\frac{1}{4}\sum e_i \cdot \zeta_*(h)(e_i)\,,
\end{equation}
where $\cdot$ denotes the Clifford multiplication in $C\ell(\mathfrak{p})$. We denote by $\gamma: Spin(\mathfrak{p}) \to End(\Delta)$ the spinor representation and write $\rho:=\gamma\circ\tilde\zeta$. The following construction is due to \cite{Partha} and has been used to calculate the eigenvalues of Dirac operators, see section \ref{diracsection}.

Let $S=G\times_\rho \Delta$ be the spinor bundle of $M=G/K$, where $G$ is viewed as a principle bundle over $M$.
$S$ splits under the action of $\mathfrak{k}$ into certain subbundles which are labeled by an index set $\mathcal{W}_0$
\begin{equation}\label{spinordec}
S=\bigoplus_{\sigma\in\mathcal{W}_0}S_\sigma\,.
\end{equation}
Let $\Phi^+_\mathfrak{g}$ and $\Phi^+_\mathfrak{k}$ be the  the sets of  $\mathfrak{g}$-positive roots and  $\mathfrak{k}$-positive roots, respectively. We define 
$\Phi^+_\mathfrak{p}:=\Phi^+_\mathfrak{g}\setminus \Phi^+_\mathfrak{k}$ and
\begin{equation}\label{halfsums}
\alpha_\mathfrak{g}:=\frac{1}{2}\sum_{\alpha\in\Phi^+_\mathfrak{g}}\alpha,\quad
\alpha_\mathfrak{k}:=\frac{1}{2}\sum_{\alpha\in\Phi^+_\mathfrak{k}}\alpha\,.
\end{equation} 
Let $\mathcal{W}$ be the Weyl group of $\mathfrak{g}$. Then $\mathcal{W}_0$ is given by
\begin{equation}
\mathcal{W}_0=\big\{\sigma\in\mathcal{W}|\,\Phi^+_\mathfrak{k}\subset \sigma\Phi^+_\mathfrak{g} \big\}\,.
\end{equation}
The spinor representation $\rho$ decomposes as 
$\rho=\sum_{\sigma\in\mathcal{W}_0}\rho_{\sigma}$,
where 
\begin{equation}\label{Pat}
\beta_\sigma =\sigma\alpha_\mathfrak{g}-\alpha_\mathfrak{k}=\frac{1}{2}\sum_{\alpha\in\sigma\Phi^+_\mathfrak{p}}\alpha
\end{equation}
 is the highest weight of the irreducible representation $\rho_{\sigma}$ of $K$. The latter appears with multiplicity one in the sum. The corresponding representation space $S_{\sigma}$ is of dimension
\begin{equation}
\dim S_{\sigma} 
=\sum_{\alpha\in\Phi^+_\mathfrak{k}}\frac{\langle \beta_\sigma+\alpha_\mathfrak{k},\alpha \rangle}{\langle\alpha,\alpha\rangle} 
=\sum_{\alpha\in\Phi^+_\mathfrak{k}}\frac{\langle \sigma\alpha_\mathfrak{g} ,\alpha \rangle}{\langle\alpha,\alpha\rangle} \,.
\end{equation}

\begin{example}
We consider $\mathfrak{g}=\mathfrak{so}(2n+1) \supset \mathfrak{so}(2n)=\mathfrak{k}$. 
The embedding is due to  the usual $\mathbf{(2n\!+\!1)}=\mathbf{2n}\,\oplus\,\mathbf{1}$ splitting. 
The roots of $\mathfrak{so}(2n+1)$ are $\big\{\pm e_k,\pm e_i\pm e_j\big\}_{1\leq k\leq n,1\leq i<j\leq n}$ and those of $\mathfrak{so}(2n)$ are given by the subset $\big\{\pm e_i \pm e_j \big\}_{1\leq i<j\leq n}$. 
The positive roots are 
$\Phi^+_{\mathfrak{so}(2n+1)}=\big\{e_k ,e_i\pm e_j \big\}_{1\leq k\leq n, 1\leq i<j\leq n}$ and 
$\Phi^+_{\mathfrak{so}(2n)}=\big\{e_i\pm e_j\big\}_{1\leq i<j\leq n}$ respectively. So 
$\Phi^+_\mathfrak{p}=\big\{e_k\big\}_{1\leq k\leq n}$. 
In particular 
$\alpha_{\mathfrak{so}(2n+1)}=\sum_{i=1}^{n}(n-i+\frac{1}{2})e_i$ and 
$\alpha_{\mathfrak{so}(2n)}=\sum_{i=1}^{n}(n-i)e_i $. 
Beside the identity, the only Weyl reflection, $\sigma$, which obeys 
$\sigma\Phi^+_{\mathfrak{so}(2n+1)}\supset\Phi^+_{\mathfrak{so}(2n)}$ 
is the one associated to the root $e_n$. 
This yields $\sigma\Phi_\mathfrak{p}^+=\big\{e_k,-e_n\,|\,k=0,\ldots n-1\big\}$. 
Therefore, the spinor representation decomposes into two summands associated to the two highest weights
\[
\tfrac{1}{2}\big(e_1+\cdots e_{n-1}\pm e_n\big)
\]
as expected.
\end{example}

A  generalization of the above construction by Parthasarathy is given by the following nice observation, see \cite{FeganSteer} or \cite{Goette1}.
\begin{proposition}\label{symmnotequalrk}
Let $M=G/K$ be a symmetric space and $\mathfrak{g}$, $\mathfrak{k}$ semisimple with rank difference $d=\rk(\mathfrak{g})-\rk(\mathfrak{k})$.
Let $\Pi$ be the set of highest weights of irreducible representations of $\mathfrak{k}$ which appear in the decomposition of the spinor representation of $M$. Then
\[
S=2^{[\frac{d}{2}]}\sum_{\beta\in\Pi}V_{\beta}\,.
\]
\end{proposition}

\begin{example}
Consider $M=G\times G / G$ with $\dim G=\dim M =n$, $\rk(G)=k$. In this case the isotropy representation is the adjoint action so that $V_{\mathfrak{so}(n)}=ad_\mathfrak{g}$. The rank difference is exactly the rank of $G$ and the spinor representation is
\[
S= 2^{[\frac{k}{2}]}V_{\alpha_\mathfrak{g}}\,,
\]
with $\alpha_\mathfrak{g}$ from (\ref{halfsums}) -- see \cite{King1982}. 
\end{example}

\subsection{General homogeneous spaces: the explicit construction}

The method cf.\ Par\-tha\-sarathy  (\ref{Pat}) does only work for the decomposition of the spinor bundle of symmetric spaces.
If we want to deal with general homogeneous spaces we have to calculate the decomposition directly. 
Even in the situation of section \ref{11} it is sometimes easy to do so, because we have to compare the weights of the respective algebras in a common base anyway.

We consider the homogeneous space $M=G/K$ with faithful isotropy representation $K\to SO(\mathfrak{p})$, where $\mathfrak{g}=\mathfrak{k}\oplus\mathfrak{p}$ is the decomposition of the associated Lie algebras. We suppose $M$ to be spin such that the  isotropy representation $\zeta_*:\mathfrak{k}\to\mathfrak{so}(\mathfrak{p}) $ gives rise to the spinor representation $\rho :K \to Spin(\mathfrak{p})\subset End(\Delta)$.

{\it The construction:\ }\  We decompose $ad_\mathfrak{g}$ with respect to $\mathfrak{k}$ and get 
\begin{equation}
ad_\mathfrak{g}= ad_\mathfrak{k}\oplus \bigoplus_\alpha V_\alpha\,,
\end{equation}
where $\bigoplus_\alpha V_\alpha$ is the decomposition of $\zeta_*=ad^\mathfrak{g}\big|_{\mathfrak{k},\mathfrak{p}}$. From the construction we see that this coincides with the vector representation of $\mathfrak{so}(\mathfrak{p})$:
\begin{equation}
\text{vector representation of }\mathfrak{so}(\mathfrak{p}) \cong \bigoplus_\alpha V_\alpha\,.
\end{equation}
This information encodes the inclusion due to the following observation, see \cite[sec.~7.6]{WybourneBook}
\begin{proposition}
Let $\iota:\mathfrak{h}\to\mathfrak{g}$ be an injective homomorphism of Lie algebras with $\mathfrak{h}$ semisimple. 
Consider the irreducible vector representation of $\mathfrak{g}$. 
Then the knowledge of its decomposition into irreducible representations with respect to $\mathfrak{h}$ 
yields the knowledge of the decomposition of any irreducible representation of $\mathfrak{g}$ with respect to $\mathfrak{h}$.
\end{proposition} 

This proposition is related to the notion of plethyms  and may be formulated as follows. 
Let $\lambda_{\mathfrak{g}}$ be the vector representation of $\mathfrak{g}$, then
\[
\lambda_{\mathfrak{g}}	
	\searrow	 \sum\limits_i \lambda^{i}_{\mathfrak{h}}
\  \Rightarrow \ 
\mu_{\mathfrak{g}}
	\searrow	\sum\limits_j \mu^j_{\mathfrak{h}} = \sum\limits_i \lambda^i_{\mathfrak{h}}\bar{\otimes}\mu_{\mathfrak{g}}\,,
\]
where $\bar{\otimes}$ denotes the plethym. For a review on plethyms see, for example, \cite{LittlewoodBook} or \cite{CastilhoAlcaras}. 
This yields that we are theoretically able to calculate  the decomposition of the spinor representation of each homogeneous space. 
Nevertheless the decomposition contains many  practical difficulties. We will illustrate the explicit construction at a well known example of low rank.

\subsection{Example: Berger space ${SO(5)/SO(3)}$}  

Let $M$ be the seven dimensional Berger space. This is the  homogeneous space $SO(5)/SO(3)$, where the subgroup is characterized as follows. We consider the embedding of $\mathfrak{so}(3)$ in $\mathfrak{so}(5)$  such that the 
five dimensional vector representation  of $\mathfrak{so}(5)$ stays irreducible. 

We recall the weights of the five and four dimensional irreducible representations of $\mathfrak{so}(5)$ and $\mathfrak{so}(3)$ and the spaces associated to these representations.
\[
{\renewcommand{\arraystretch}{1.3}
\begin{array}{c|c|c|c|c}
\dim \text{of}& \multicolumn{2}{c|}{\mathfrak{so}(5)} 		& \multicolumn{2}{c}{\mathfrak{so}(3)}\\
\text{Rep.} & \text{highest weight} 		& \text{space} 	& \text{highest weight} & \text{space} \\\hline
\mathbf{5} & (1,0)=e_1 			& \r^5 		& (4)=4\lambda 	& S^2_0(\r^3) \\\hline
\mathbf{4} & (\tfrac{1}{2},\tfrac{1}{2})=\frac{1}{2}(e_1+e_2)	
					& S_{\frac{1}{2}}& (3)=3\lambda	& S_{\frac{3}{2}}
\end{array} }
\]
Here we write $\lambda$ for the highest weight of the vector representation of $\mathfrak{so}(3)$ so that the weights of $\mathfrak{so}(3)$-representations are given by $k\lambda$. 

The root system of $\mathfrak{so}(5)$ is given by $\{\pm e_1, \pm e_2,\pm e_1\pm e_2\}$.
 
To give $\lambda$ in terms of $\{e_1,e_2\}$ we use that the weight diagram of $\mathbf{5}$ 
with respect to $\mathfrak{so}(5)$ projects onto the corresponding weight diagram with respect to $\mathfrak{so}(3)$. 
The weight diagram  of $\mathbf{5}$ with respect to $\mathfrak{so}(5)$ is $\{\pm e_1,\pm e_2,0\}$. Up to symmetries we have
\begin{equation}
\begin{split}
\text{proj}_{span\{\lambda\}}\big(e_1\big)&= 4\lambda\,,\\
\text{proj}_{span\{\lambda\}}\big(e_2\big)&= 2\lambda\,,\\
\end{split}
\end{equation}
or
\begin{align*}
 \langle e_1,\lambda \rangle \frac{\lambda}{|\lambda|^2} =4\lambda   \quad &\wedge \quad 
 \langle e_2,\lambda \rangle \frac{\lambda}{|\lambda|^2} =2\lambda  \,, 
\\
 \lambda_1=4|\lambda|^2    \quad &\wedge \quad 
 \lambda_2=2|\lambda|^2    
 \,.
\end{align*}
From the last line we get  $|\lambda|^2=\frac{1}{20}$ and so
\begin{equation}
\lambda=\tfrac{1}{10}\big(2e_1+e_2\big)\,.
\end{equation}
\begin{remark}
The $\mathbf{4}$ of $\mathfrak{so}(5)$ stays irreducible as well. Its weight diagram is given by $\{\pm e_1\pm e_2\}$ and we have
$\text{proj}_{span\{\lambda\}}\big(\tfrac{1}{2}(e_1\pm e_2)\big)= (2\pm1)\lambda$.
\end{remark}
To  decompose the adjoint of $\mathfrak{so}(5)$  with respect to $\mathfrak{so}(3)$ we need the projection\footnote{We omit the projection of the origin, because its multiplicity stays the same.} of the weight diagram on $span\{\lambda\}$:
\begin{align*}
\text{proj}_{span\{\lambda\}}(e_1+e_2) & =6\lambda\,,\\
\text{proj}_{span\{\lambda\}}(e_1) & =4\lambda\,,\\
\text{proj}_{span\{\lambda\}}(e_2) & =2\lambda\,,\\
\text{proj}_{span\{\lambda\}}(e_1-e_2) & =2\lambda\,.
\end{align*}
We see that the image contains two diagrams to the highest weight $6\lambda$ and $2\lambda$, respectively, where $2\lambda$ is  the adjoint of $\mathfrak{so}(3)$ and $6\lambda$ the seven dimensional irreducible representation. Therefore
\[
\mathfrak{so}(5)=\mathfrak{so}(3)\,\oplus\,\mathbf{7}\,.
\]
The isotropy representation yields the embedding $\mathfrak{so}(3)\hookrightarrow\mathfrak{so}(7)$ and the calculations above show that this remains irreducible with respect to $\mathfrak{so}(3)$. 

To get the decomposition of the spinor representation, we need the weight diagram of $\mathfrak{so}(3)$ as subset of the weight lattice of $\mathfrak{so}(7)$. Therefore we turn to a three dimensional picture and write the  roots of $\mathfrak{so}(7)$ as $\big\{\pm e_1,\pm e_2,\pm e_3,\pm e_1\pm e_2,\pm e_1\pm e_3,\pm e_2\pm e_3\big\}$ such that the weight lattice of $\mathbf{7}$ is given by $\big\{\pm e_1,\pm e_2,\pm e_3,0\big\}$. As before we denote the highest weight of the vector representation of $\mathfrak{so}(3)$ as $\lambda$ so that the seven dimensional representation is given by $6\lambda$. We get
\begin{equation}\begin{split}
\text{proj}_{span\{\lambda\}}(e_1)&=6\lambda\,,\\
\text{proj}_{span\{\lambda\}}(e_2)&=4\lambda\,,\\
\text{proj}_{span\{\lambda\}}(e_3)&=2\lambda\,,\\
\end{split}\end{equation}
or
\begin{equation}
\lambda= \tfrac{1}{56}(6,4,2)\,.
\end{equation}
The spin representation is eight dimensional and the weights are given by $\big\{\frac{1}{2}(\pm e_1 \pm e_2 \pm e_3)\big\}$. The projection on the weight lattice of $\mathfrak{so}(3)$ is
\begin{equation*}
\text{proj}_{span\{\lambda\}}\big(\big\{\frac{1}{2}(\pm e_1\pm e_2 \pm e_3)\big\}\big) = 
\big\{ \pm6\lambda, \pm4\lambda,\pm2\lambda,0\lambda\big\}\cup\big\{0\lambda\big\}\,,
\end{equation*}
which yields 
\begin{proposition}
The spinor representation of $M=SO(5)/SO(3)$ split into the $(6)$ and $(0)$ of $\mathfrak{so}(3)$:
\begin{equation}
\mathbf{8}=\mathbf{7}\oplus\mathbf{1}\,.
\end{equation}
\end{proposition}
In particular the Berger space  is a nearly parallel $G_2$-Einstein manifold, in particular it  admits a connection which annihilates one spinor (see \cite{BesseBook} and  \cite{FriedIvanov3}).

\section{The spinor bundle of Grassmann manifolds}\label{sectionX}
We recall the following observation (see\cite{CahenGutt} or \cite{Strese1}).
\begin{proposition}
The Grassmannian $G_{m,n}=SO(n+m)/SO(n)\times SO(m)$ is spin if and only if $m=1$ or $n=1$ or $m+n$ even.
\end{proposition}
In particular if the Grassmannian is not of type $G_{n,n-1}$ all even and odd dimensional spin Grass\-manianns are of the form 
$G_{2k,2\ell}$ and $G_{2k+1,2\ell+1}$, respectively.
Therefore we divide this section into two parts dedicated to the even and odd dimensional spin Grassmannians, respectively.

\subsection{The even dimensional case}\label{subsectionX}

We consider the even dimensional spin Grassmannian $G_{2k,2\ell}=SO(2(k+\ell))/ (SO(2k)\times SO(2\ell))$. Our goal is to  decompose the spinor representation of $G_{2k,2\ell}$ with respect to $\mathfrak{so}(2k)\oplus \mathfrak{so}(2\ell)$. We consider $\ell\leq k$ and  restrict ourselves to $\ell\neq 1$ because this case is treated in detail in \cite{Strese1}. We have
\begin{equation}
\mathfrak{so}(2(k+\ell))=\mathfrak{so}(2k)\oplus \mathfrak{so}(2\ell)\oplus (\mathbf{2k}\otimes\mathbf{2\ell})\,.
\end{equation}
This yields that the isotropy representation  $\mathfrak{so}(2k)\oplus\mathfrak{so}(2\ell)\hookrightarrow\mathfrak{so}(4k\ell)$ is the standard embedding. This means that the vector representation of $\mathfrak{so}(4k\ell)$ decomposes as
\begin{equation}
\mathbf{4k\ell}=\mathbf{2k}\otimes\mathbf{2\ell}\,.
\end{equation}
Moreover we know that the adjoint representation decomposes as follows
\begin{equation}
{\mathfrak{so}(4k\ell)}={\mathfrak{so}(2k)}\oplus {\mathfrak{so}(2\ell)}
		\oplus \big({\mathfrak{so}(2k)}\otimes S^2_0(\r^{2\ell})\big)
		\oplus \big(S^2_0(\r^{2k})\otimes  {\mathfrak{so}(2\ell)}\big)\,.
\end{equation}
We construct a basis of the Cartan algebra of $\mathfrak{so}(4k\ell)$ such that we recover the Cartan basis of $ \mathfrak{so}(2k)$ and $ \mathfrak{so}(2\ell)$. 
We denote the Cartan basis of $\mathfrak{so}(2k)$ and $\mathfrak{so}(2\ell)$  by $\{K_i\}_{1\leq i\leq k}$ and $\{L_i\}_{1\leq i\leq \ell}$, respectively. The associated decompositions into two dimensional subspaces are $\r^{2k}=V_1\oplus\cdots\oplus V_k$ and $\r^{2\ell}=W_1\oplus \cdots\oplus W_\ell$ with $V_i=span\{v^i_1,v_2^i\}$, $W_j=span\{w_1^j,w_2^j\}$, i.e., $K_i(v^j_1)=\delta_i^j v^i_2$, $K_i(v^j_2)=-\delta_i^j v^i_1$ and similar for $L_i$.

We write $V_i\otimes W_j = E_{ij}\oplus F_{ij}$ with $E_{ij}=span\{e^{ij}_+,e^{ij}_-\}$, $F_{ij}=span\{f^{ij}_+,f^{ij}_-\}$ and
\begin{equation}
\begin{split}
   	e^{ij}_+=\tfrac{1}{\sqrt{2}}(v^i_1\otimes w^j_1+v^i_2\otimes w^j_2)\,,
\qquad 	e^{ij}_-=\tfrac{1}{\sqrt{2}}(v^i_2\otimes w^j_1-v^i_1\otimes w^j_2)\,, \\
   	f^{ij}_+=\tfrac{1}{\sqrt{2}}(v^i_1\otimes w^j_2+v^i_2\otimes w^j_1)\,,
\qquad	f^{ij}_-=\tfrac{1}{\sqrt{2}}(v^i_1\otimes w^j_1-v^i_2\otimes w^j_2) \,.
\end{split}
\end{equation}
We define $\{N^e_{ij}, N^f_{ij}\}_{1\leq i\leq k,1\leq j\leq\ell}$  by 
\begin{equation}
\begin{split}
N^e_{ij}(e^{i'j'}_\pm)&=\pm \delta^{i'}_i\delta^{j'}_je^{ij}_\mp\,,\\
N^f_{ij}(f^{i'j'}_\pm)&=\pm \delta^{i'}_i\delta^{j'}_j f^{ij}_\mp\,.
\end{split}
\end{equation}
This is the Cartan basis with associated decomposition 
\begin{equation}
\r^{4k\ell}=\bigoplus_{i,j}(V_i\otimes W_j) =\bigoplus_{i,j} E_{ij} \oplus\bigoplus_{i,j}F_{ij}\,.
\end{equation}
We have 
\begin{alignat*}{2}
	K_i(e^{i'j}_\pm)&=\pm\delta^{i'}_i e^{ij}_\mp\,,
&\ 	K_i(f^{i'j}_\pm)&=\mp\delta^{i'}_i f^{ij}_\mp\,,
\\
	L_j(f^{ij'}_\pm)&=\mp \delta^{j'}_j f^{ij}_\mp\,,
&\ 	L_j(e^{ij'}_\pm)&=\mp \delta^{j'}_j e^{ij}_\mp\,,
\end{alignat*}
for all $i=1,\ldots,k$ and $j=1,\ldots,\ell$ , such that
\begin{equation}
K_i =\sum_{j=1}^\ell (N^e_{ij}-N^f_{ij})\ \text{and}\ 
L_j =-\sum_{i=1}^k (N^e_{ij}+N^f_{ij})\,.
\end{equation}
We write $\epsilon^x_{ij}=(N^x_{ij})^*$ such that the roots of $\mathfrak{so}(4k\ell)$ are given by 
\begin{equation}
\big\{\pm \epsilon^x_{ij}\pm\epsilon^y_{i'j'}\,|\, x,y\in\{e,f\};1\leq i,i'\leq k;1\leq j,j'\leq\ell \big\}\,.
\end{equation} 
The roots which form the subalgebra $\mathfrak{so}(2k)\oplus\mathfrak{so}(2\ell)$   are given by $\{\pm K^*_i\pm K^*_{i'}\}$ and $\{\pm L^*_j\pm L^*_{j'}\}$ or 
\begin{multline}\label{rootsH}
         \Big\{ \sum_{j=1}^\ell\big( \pm(\epsilon_{ij}^e -\epsilon_{ij}^f)
		\pm (\epsilon_{i'j}^e-\epsilon_{i'j}^f)\big)\,|\, 1\leq i,i'\leq k\Big\} 
\\
\cup   \Big\{\sum_{i=1}^k\big(\pm(\epsilon_{ij}^f +\epsilon_{ij}^e)
		\pm(\epsilon_{ij'}^f+\epsilon_{ij'}^e)\big)\,|\, 1\leq j,j'\leq \ell \Big\}\,.
\end{multline}
To get the decomposition of the spinor representation of $\mathfrak{so}(4k\ell)$ with respect to $\mathfrak{so}(2k)\oplus\mathfrak{so}(2\ell)$ we project the diagram\footnote{The diagram splits into two diagrams which correspond to the positive, respectively negative spinor representations depending on whether the number of minus signs is even, respectively odd.} $\frac{1}{2}(\pm1,\ldots,\pm1)=\frac{1}{2}\sum_{x,i,j}(\pm\epsilon^{x}_{ij})$ onto the subspace spanned by (\ref{rootsH}) and expand the result with respect to $\{K_i^*,L_j^*\}$. 
Explicitly this is done by writing 
\[
\begin{bmatrix}
	K_i^*\\ 
	L_j^*
\end{bmatrix}
= A
\begin{bmatrix}
\epsilon^e_{ij}\\\epsilon^f_{ij}
\end{bmatrix}\,.
\]
If we choose $\epsilon^x_{ij}=(\epsilon^x_{11},\ldots,\epsilon^x_{1\ell},\ldots,\epsilon^x_{k1},\ldots,\epsilon^x_{k\ell})$ the 
$[(k+\ell)\times2k\ell]$-matrix $A$ is given by the following rows:
\begin{equation}\label{matrix}
\begin{aligned}
A_i= \big(
\overbrace{0,\ldots, 0}^{(i-1)\ell},
\overbrace{1\ldots,1}^{\ell},
\overbrace{0,\ldots,0}^{(k-i)\ell},
\overbrace{0,\ldots, 0}^{(i-1)\ell},
\overbrace{-1\ldots,-1}^{\ell},
\overbrace{0,\ldots,0}^{(k-i)\ell}
\big)\,, &\quad 1\leq i\leq k\,,\\
A_{k+j}= \big(
\underbrace{0,\ldots,0,\overset{\overset{j}{\downarrow}}{-1},0,\ldots 0}_{\ell},\ldots,
\underbrace{0,\ldots,0,\overset{\overset{j+(2k-1)\ell}{\downarrow}}{-1},0,\ldots,0}_{\ell}
\big)\,, &\quad 1\leq j\leq \ell\,.
\end{aligned}
\end{equation}
The projected diagram is read from 
\begin{equation}\label{images}
\tfrac{1}{2} A \big(\pm1,\ldots,\pm1\big)^T
\end{equation}
and we get the first observation.
\begin{remark}
All the images of the set $ \{ \frac{1}{2}(\pm1,\ldots,\pm1\big)^T\}$ by the map $A$ consist of weights with integer entries.
\end{remark}
To illustrate this procedure we will examine the examples $k=\ell=2$,  $k=\ell=3$, and  $k=4,\ell=2$ before we state the general result.

\underline{$k=\ell=2$:} 
The matrix $A$ with $\big(4K^*_i,4L^*_j\big)^T=A\big(\epsilon^e_{ij},\epsilon^f_{ij}\big)^T$  is explicitly given by 
\begin{equation*}
{ \setlength{\arraycolsep}{3pt}
A=
\left[\begin{array}{cccccccc}
1&1&&&-1&-1&&\\
&&1&1&&&-1&-1\\
-1&&-1&&-1&&-1&\\
&-1&&-1&&-1&&-1
\end{array}\right]\,.
}
\end{equation*}
To get the image of a vector $\vec{x}\in span\{\epsilon^x_{ij}\}$ under the projection we need  $A\vec{x}$. For $\vec{x}$ contained in the spinor diagram the images are, for example,
\begin{align*}
\tfrac{1}{2}(-1,\mp1,-1,\mp1,-1,\mp1,-1,\mp1)& \longmapsto  ((0,0),(2,\pm2))\,, \\
\tfrac{1}{2}(+1,+1,\pm1,\pm1,-1,-1,\mp1,\mp1)& \longmapsto  ((2,\pm2),(0,0))\,, \\
\tfrac{1}{2}(-1,+1,\mp1,\pm1,-1,-1,\mp1,\mp1)& \longmapsto  ((1,\pm1),(2,0) )\,, \\
\tfrac{1}{2}(+1,+1,-1,\mp1,-1,-1,-1,\mp1)& \longmapsto  ((2,0),(1,\pm1))\,, \\
\tfrac{1}{2}(-1,+1,-1,\mp1,-1,-1,-1,\mp1)& \longmapsto  ((1,0),(2,\pm1))\,, \\
\tfrac{1}{2}(+1,+1,-1,\pm1,-1,-1,-1,\mp1)& \longmapsto  ((2,\pm1),(1,0))\,. 
\end{align*}
$(2,\pm2)$ may be identified with the subset of trace-free 4-tensors on $\r^4$ with  symmetry of the Young diagram $\tiny \yc\yng(2,2)$. 
$(2,\pm1)$ is given by the subset of trace-free 3-tensors with symmetry $\tiny \yc\yng(2,1)$. In both cases $\pm$ indicates the eigenspaces of the symmetry of the tensors, which is induced by the self duality of two forms in dimension 4. 

In table \ref{table1} we list  the representation spaces, associated Young diagrams, and add the dimension as well as the further decomposition with respect to $\mathfrak{su}(2)\oplus\mathfrak{su}(2)$. The notation has been taken from \cite{McKayPatera} and contains the highest weights of the two factors, e.g., $(3|1)=\mathbf{4}\otimes\mathbf{2}$.
\begin{table}[htb] \caption{Representations of $\mathfrak{so}(4)$}\label{table1}
\[
{\renewcommand{\arraystretch}{1.5}
\begin{array}{c|c|c|c}
\text{Rep.\ of }\mathfrak{so}(4) & \text{Symmetry} 
				& \text{Dec.\ wrt.\ }\mathfrak{su}(2)\oplus\mathfrak{su}(2)
							& \text{Dimension} \\\hline
(00) 	& \mathbf{\cdot}		& (0|0)			&1\\
(10) 	& \tiny\yng(1)		& (1|1)			&4\\
(11^\pm)	& \tiny\yng(1,1)		& (2|0)\text{ and }(0|2)	&2\cdot3\\
(20) 	& \tiny\yng(2)_0		& (2|2)			&9\\
(22^\pm)	& \tiny\yng(2,2)_0		& (4|0)\text{ and }(0|4)	&2\cdot5\\
(21^\pm)	& \tiny\yng(2,1)_0		& (3|1)\text{ and }(1|3)	&2\cdot8 
\end{array}
}
\]
\end{table}

So the decomposition of the spinor representation into (reducible) representation spaces with respect to $\mathfrak{so}(4)\oplus\mathfrak{so}(4)$ is
\begin{equation}\label{spinor22}
\begin{split}
S^+ &=  
\big(\mathbf{1}\otimes\mathbf{10}\big) \oplus\big( \mathbf{10}\otimes\mathbf{1} \big)
	\oplus \big(\mathbf{6}\otimes\mathbf{9}\big) \oplus\big( \mathbf{9}\otimes\mathbf{6} \big)\,,\\
S^- &= \big(\mathbf{4}\otimes\mathbf{16}\big) \oplus\big( \mathbf{16}\otimes\mathbf{4} \big)\,.
\end{split}
\end{equation}
and the irreducible decomposition -- or equivalently the decomposition  with respect to $\mathfrak{su}(2)\oplus\mathfrak{su}(2)\oplus\mathfrak{su}(2)\oplus\mathfrak{su}(2)$ --  is
\begin{equation}
\begin{split}
S^+ &= (0|0|0|4)\oplus(0|0|4|0)\oplus(0|4|0|0)\oplus(4|0|0|0) \\
	&\qquad \oplus(2|2|2|0)\oplus (2|2|0|2)\oplus (2|0|2|2)\oplus (0|2|2|2)\,, \\
S^- &= (1|1|1|3)\oplus (1|1|3|1)\oplus (1|3|1|1) \oplus (3|1|1|1) \,.
\end{split}
\end{equation}

\underline{$k=4$, $\ell=2$:} 
The matrix with $ \big(4K_i^*,8L_i^*\big)^T =A \big(\epsilon^e_{ij},\epsilon^f_{ij}\big)^T $ 
is explicitly given by 
\begin{equation*}
{ \setlength{\arraycolsep}{3pt}
\left[
\begin{array}{cccccccccccccccc}
1&1&&&&&& &-1&-1&&&&&&\\
&&1&1&&&& & &&-1&-1&&&&\\ 
&&&&1&1&& &&&&&-1&-1&&\\
&&&&&&1&1&&&&&&&-1&-1\\
-1&&-1&&-1&&-1&&-1&&-1&&-1&&-1&\\
&-1&&-1&&-1&&-1&&-1&&-1&&-1&&-1
\end{array}\right]\,.
}
\end{equation*}
The representation spaces which are associated to the images of $\tfrac{1}{2}(\pm1,\ldots,\pm1)$ are listed in table \ref{table2}.
They are divided  into two parts such that the first 18 summands give the decomposition of the positive spinor representation, and the second 12  summands yield the decomposition of the negative spinor representation.
\begin{table}[htb]\caption{Representations of $\mathfrak{so}(8)\oplus\mathfrak{so}(4)$}\label{table2}
\[
{\renewcommand{\arraystretch}{1.5}
\begin{array}{c|c|r@{\ =\ }l}
\text{Representation of} 
					&\text{Irreducible Dec.\ w.r.t.}
								&\multicolumn{2}{c}{} 	\\
\mathfrak{so}(8)\oplus \mathfrak{so}(4) 
					&\mathfrak{so}(8)\oplus\mathfrak{su}(2)\oplus\mathfrak{su}(2)
								& \multicolumn{2}{c}{\text{Dimension}}	\\\hline\hline
(0000|44^\pm)	&
  	(0000|8|0)\oplus(0000|0|8)   		& 2\cdot(1\cdot9\cdot1)	& 2\cdot 9 \\
(2000|33^\pm)	&
 	 (2000|6|0)\oplus(2000|0|6)   		& 2\cdot(35\cdot7\cdot1) 	& 2\cdot 245\\
(1100|42^\pm)	&
 	 (1100|6|2) \oplus(1100|2|6)   	& 2\cdot(28\cdot7\cdot3)	& 2\cdot 588\\
(2200|22^\pm)	&
 	 (2200|4|0) \oplus(2200|0|4)    	& 2\cdot(300\cdot5\cdot1) 	& 2\cdot 1500\\
(2110|31^\pm)	&
	 (2110|4|2) \oplus(2110|2|4)    	& 2\cdot(350\cdot5\cdot3)	& 2\cdot 5250\\
(1111^\pm|40)	&
	 (1111^+|4|4) \oplus(1111^-|4|4)    	& 2\cdot(35\cdot5\cdot5)	& 2\cdot 875\\
(2220|11^\pm)	&
 	 (2220|2|0) \oplus(2220|0|2)    	& 2\cdot(840\cdot3\cdot1)	& 2\cdot 2520\\
(2211^\pm|20)	&
	 (2211^+|2|2) \oplus(2211^-|2|2)    	& 2\cdot(567\cdot3\cdot3)	& 2\cdot 5103\\
(2222^\pm|00)	&
	 (2222^+|0|0) \oplus(2222^-|0|0)    	& 2\cdot(294\cdot1\cdot1)	& 2\cdot 294\\ \hline 
(1000|43^\pm)	&
	 (1000|7|1) \oplus(1000|1|7)  	  	& 2\cdot(8\cdot8\cdot2)	& 2\cdot 128\\
(2100|32^\pm)	&
	 (2100|5|1) \oplus(2100|1|5)  	  	& 2\cdot(160\cdot6\cdot2)	& 2\cdot1920\\
(1110|41^\pm)	&
	 (1110|5|3) \oplus(1110|3|5)  	  	& 2\cdot(56\cdot6\cdot4)	& 2\cdot1344\\
(2210|21^\pm)	&
	 (2210|3|1) \oplus(2210|1|3)  	  	& 2\cdot(840\cdot4\cdot2)	& 2\cdot 6720\\
(2111^\pm|30)	&
	 (2111^+|3|3) \oplus(2111^-|3|3)    	& 2\cdot(224\cdot4\cdot4)	& 2\cdot3584\\
(2221^\pm|10) 	&
	 (2221^+|1|1) \oplus(2221^-|1|1)    	& 2\cdot(672\cdot2\cdot2)	& 2\cdot 2688\\\cline{4-4}
\multicolumn{2}{c}{}			&			& 2^{16}	
\end{array}
}
\]
\end{table}

\underline{$k=\ell=3$:} 
In this example the matrix which obeys $(6K_i^*,6L_j^*)=A(\epsilon^e_{ij},\epsilon^f_{ij})$ has size $6\times18$.
As before, we list the irreducible representation spaces which are associated to the images of $\tfrac{1}{2}(\pm1,\ldots,\pm1)$. The result for $S^+$ can be found in table \ref{table3}. The representations for $S^-$ can be obtained by interchanging the two factors.

\begin{table}[htb]\caption{Representations of $\mathfrak{so}(6)\oplus\mathfrak{so}(6)$}\label{table3}
\[
{\renewcommand{\arraystretch}{1.5}
\begin{array}{c|r@{\ =\ }l}
\text{Irreducible Rep.\ of}
					&\multicolumn{2}{c}{}\\
\mathfrak{so}(6)\oplus\mathfrak{so}(6)
					& \multicolumn{2}{c}{\text{Dimension}} \\\hline\hline
 (000|333^\pm)	& 	2\cdot(1\cdot84)		&2\cdot 84\\
 (110|331^\pm)	& 	2\cdot(15\cdot270)		&2\cdot 4050\\
 (200|322^\pm)	& 	2\cdot(20\cdot140)		&2\cdot2800 \\
 (211^\pm|320)	& 	2\cdot(45\cdot300)		&2\cdot 13500\\
 (220|311^\pm)	& 	2\cdot(84\cdot126)		& 2\cdot10584 \\  
 (222^\pm|300)	&	 2\cdot(35\cdot50)		&2\cdot 1750 \\
 (332^\pm|100)	& 	2\cdot(6\cdot189)		&2\cdot1134  \\
 (330|111^\pm)	& 	2\cdot(10\cdot300)		&2\cdot3000 \\
 (321^\pm|210)	& 	2\cdot(64\cdot256)		&2\cdot16384 \\
 (310|221^\pm)	& 	2\cdot(70\cdot175)		&2\cdot12250 \\ \cline{3-3}
\multicolumn{1}{c}{}& 				&\tfrac{1}{2}\cdot 2^{18}
\end{array}
}
\]
\end{table}

Before we state the general result we introduce the following operation.
\begin{definition}
Let $\lambda=(\lambda_1, \ldots,\lambda_n)$ be a vector with non-negative decreasing integer entries. 
For $m\geq\max\{\lambda_j\},\ell\geq n$ the $(\ell,m)$-conjugate\footnote{We may always assume $n=\ell$ by extending $\lambda$ by zeros.} $\lambda^{c(\ell,m)}$  is defined by the vector which represents the Young  diagram obtained by the following procedure. Extend the Young diagram associated to $\lambda$ to an rectangle of size $(\ell\times m)$, erase $\lambda$, rotate the remaining part by $180^\circ$, and reflect at the main diagonal. 
\[
(\lambda^{c(\ell,m)})_j =\ell - \#\{i | \lambda_i\geq m-j+1\} \qquad\text{for }  1\leq j\leq m\,.
\]
\end{definition}
For example, $(5,3,2,0)^{c(4,6)}=(4,3,3,2,1,1)$:
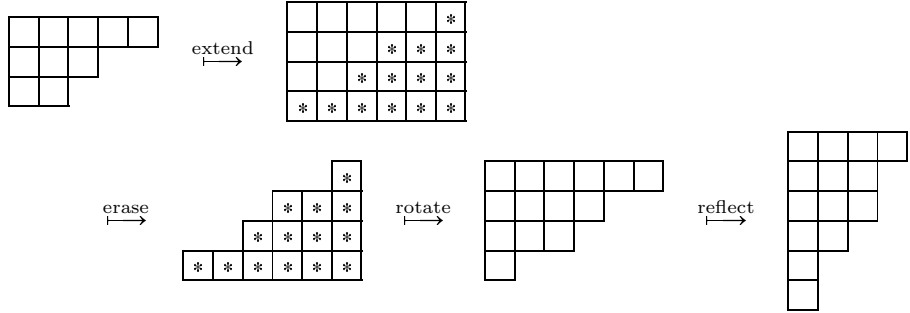
\begin{multline*}
\yc \yng(5,3,2)
\quad\overset{\text{extend}}{\longmapsto}\quad
\young(\hfil\hfil\hfil\hfil\hfil*,\hfil\hfil\hfil***,\hfil\hfil****,******) \\
\yc \quad\overset{\text{erase}}{\longmapsto}\quad
\young(:::::*,:::***,::****,******) 
\quad\overset{\text{rotate}}{\longmapsto}\quad
\yng(6,4,3,1)
\quad\overset{\text{reflect}}{\longmapsto}\quad
\yng(4,3,3,2,1,1)
\end{multline*}
\begin{remark} \begin{itemize}
\item We denote by $\lambda'$ the transpose of the diagram $\lambda$, i.e., the reflection of $\lambda$ at the main diagonal. Then  the $(\ell,m)$-conjugate and the transpose are connected by 
\begin{equation*}
\lambda^{c(\ell,m)}_j =\ell-{\lambda'}_{m-j+1},\quad 1\leq j\leq m\,.
\end{equation*}
\item Furthermore we have 
\begin{equation*}
(\lambda')' = (\lambda^{c(\ell,m)})^{c(m,\ell)}=\lambda
\end{equation*}
and therefore
\begin{equation*}
(\lambda^{c(\ell,m)})'_j= m-\lambda_{\ell-j+1},\quad 1\leq j\leq \ell\,.
\end{equation*}
\item For $\lambda=(\lambda_1,\ldots,\lambda_\ell)$ we have 
\begin{equation}\label{lambdaC}
\lambda^{c(\ell,m)}=(  \ell^{m-\lambda_1},(\ell-1)^{\lambda_1-\lambda_2},\ldots,1^{\lambda_{\ell-1}-\lambda_\ell},0^{\lambda_\ell})\,.
\end{equation}
\end{itemize}
\end{remark}

\begin{conjecture}
\label{decspinorgrass}
Let $G_{2k,2\ell}=SO(2(k+\ell))/(SO(2k)\times SO(2\ell))$ be the spin Grass\-mannian of even dimension $4k\ell$ 
and let $S^+$ and $S^-$ the positive and negative spinor bundles respectively. Suppose $k\geq\ell$. 
The spinor bundle $S^+\oplus S^-$ decomposes into a sum of $2\!\cdot\!\binom{k+\ell}{\ell}$ subbundles 
associated to  irreducible representations of  $\mathfrak{so}(2k)\oplus\mathfrak{so}(2\ell)$. 
The weights of these representations are given by 
\begin{equation}\label{dec}
(\lambda_1,\ldots, \lambda_k  | \mu_1,\ldots, \mu_\ell)  \,,
\end{equation}
with $\ell\geq \lambda_1\geq \cdots\geq \lambda_{k-1}\geq |\lambda_k| \geq 0$, $k \geq \mu_1\geq\cdots\geq \mu_{\ell-1}\geq|\mu_\ell|\geq 0$ and
\begin{align}
	(\lambda_1,\ldots, |\lambda_k|) 	&=  (\mu_1,\ldots, |\mu_\ell|)^{c(\ell,k)}\,, \label{reps}\\
\intertext{or equivalently}
	(\mu_1,\ldots, |\mu_\ell|) 	&= (\lambda_1,\ldots, |\lambda_k|)^{c(k,\ell)}\,. \tag{\ref{reps}'}
\end{align}
In particular $\sum_{i=1}^k \lambda_i+\sum_{j=1}^\ell \mu_j=k\ell$.
\end{conjecture}
\begin{theorem}\label{decspinorgrassl2}
Conjecture \ref{decspinorgrass} is true for $\ell=2$.
\end{theorem}

\begin{remark}
\begin{itemize}
\item
For $\lambda_k\neq0$ (in this case $\mu_\ell=0$) the representation with $-\lambda_k$ is present within the decomposition, too. 
This symmetry is due to self duality with respect to $\mathfrak{so}(2k)$. 
This argument is symmetric with respect to the two factors $\mathfrak{so}(2k)$ and $\mathfrak{so}(2\ell)$ so that 
$S^\pm$ splits itself into two subbundles $S^\pm=S^{\pm+}\oplus S^{\pm-}$.
\item
If the product $k\ell$ is even, representations with weights such that both sums $\sum_{i=1}^k \lambda_i$ and $\sum_{j=1}^\ell \mu_j$ are even (or odd) belong to $S^+$ (or $S^-$ respectively). 
In particular, $\big( 0^k | k^\ell\big)$ and $\big(\ell^k | 0^\ell\big)$ belong to $S^+$.  
\item 
If the product $k\ell$ is odd, representations with $\sum_{i=1}^k \lambda_i$ even (or odd) and $\sum_{j=1}^\ell \mu_j$ odd (or even) 
belong to $S^+$ (or $S^-$ respectively). 
In particular $\big( 0^k | k^\ell \big)$ belongs to $S^+$ and $\big( \ell^k | 0^\ell \big)$ belongs to $S^-$.
\end{itemize}\end{remark}

{\em Towards a proof.\ }\ 
Firstly we have to show that all the mentioned weights appear as an image of $A$ on the set $\big\{ (\pm\frac{1}{2},\ldots,\pm\frac{1}{2}) \big\}$. Consider the $\mathfrak{so}(2\ell)$-representation $\lambda=(\lambda_1,\ldots,\lambda_\ell)$ with $\lambda_\ell\geq0$. We associate vectors $\vec{f}_\mu\in\r^{\ell}$ and $\vec{w}\in\r^{2kl}$ given by 
\[
2\vec{f}_{i}=(\overbrace{-1,\ldots,-1}^{i},1,\ldots,1), \quad 0\leq i\leq\ell
\]
and 
\begin{equation*}
\vec{w}=\big(
\underbrace{\vec{f}_0,\ldots,\vec{f}_0}_{k-\lambda_1},\ldots,
\underbrace{\vec{f}_i,\ldots,\vec{f}_i}_{\lambda_{i}-\lambda_{i+1}},\ldots,
\underbrace{\vec{f}_{\ell},\ldots,\vec{f}_{\ell}}_{k+\lambda_{\ell}}
\big)\,.
\end{equation*}
to  the tuple $\lambda$. Then $\vec{w}$ is a preimage of 
\[
\big(\ell^{k-\lambda_1},(\ell-1)^{\lambda_1-\lambda_2},\ldots, 1^{\lambda_{\ell-1}-\lambda_\ell},0^{\lambda_\ell}| \lambda_1,\ldots,\lambda_\ell\big)
\]
under the action of $A$ with rows given by (\ref{matrix}).

The prove of the statement now is done by a dimension check: 
For $\lambda$ with $\lambda_k\geq0$ we write $\tilde\lambda=\lambda^{c(k,\ell)}$ such that  $\tilde{\tilde\lambda}=\lambda$. 
The dimension of the $\mathfrak{so}(2k)$-irrep  with highest weight $\lambda=(\lambda_1,\ldots,\lambda_k)$ 
and the corresponding dimension of the $\mathfrak{so}(2\ell)$-irrep with highest weight $\tilde\lambda=(\tilde\lambda_1,\ldots,\tilde\lambda_\ell)$ 
are given by
\begin{align*}
\dim V^{\mathfrak{so}(2k)}_\lambda 	& = \prod_{1\leq i<j\leq k}\frac{(\lambda_i+k-i)^2-(\lambda_j+k-j)^2}{(k-i)^2-(k-j)^2}\\
             	& = \prod_{1\leq i<j\leq k}\frac{(\lambda_i+\lambda_j+2k-i-j)(\lambda_i-\lambda_j-i+j)}{(2k-i-j)(j-i) } 
\end{align*}
and $(\lambda,k)$ substituted by $(\tilde\lambda,\ell)$.
The dimension may also be expressed using the following determinant
\begin{align*}
   d(\lambda_1,\ldots, \lambda_k)
	& = \det\left[ 
    		    \binom{2k+\lambda_i-i+j-1}{2k-1}-\binom{2k+\lambda_i-i-j-1}{2k-1}
	       \right]_{1\leq i,j\leq  k}\,,
\end{align*}
see \cite{FultonHarris}. To prove the main statement we have to show
\begin{equation}\label{master}
\sum_{\lambda\subset (\ell^k)} \dim V^{\mathfrak{so}(2k)}_{\lambda}\cdot\dim V^{\mathfrak{so}(2\ell)}_{\tilde\lambda}=2^{2k\ell-1}\,,
\end{equation}
where we have to take one half of the dimension of the spinor module because the construction yields that either $\lambda$ or $\tilde\lambda$ has vanishing last component.

\subsection{The case $\ell\leq 2$}

In this section we prove theorem \ref{decspinorgrassl2}. We recall that conjecture \ref{decspinorgrass} for $\ell=1$ is shown in \cite{Strese1} such that we consider the case $G_{2k, 4}$, which is one of the compact quaternionic K\"ahler symmetric spaces classified in \cite{Wolf1}. 

We label the $\mathfrak{so}(4)$-weights by  two numbers $k\geq p\geq q\geq 0$ such that the spaces in the  decomposition (\ref{dec}) have the weights 
\begin{equation}\label{dec2}
(2^{k-p}1^{p-q}0^q|pq)^\pm\,.
\end{equation}

The dimension of the $\mathfrak{so}(4)$-representation $(p,q)$ is 
\begin{equation*}
\dim V^{\mathfrak{so}(4)}_{(pq^\pm)} =(p+q+1)(p-q+1)\,.
\end{equation*}

We use the classical result of \cite{LittlewoodBook} on the characters of the classical groups to evaluate the dimensions of the representation spaces of $\mathfrak{so}(2k)$ which we need here.
\begin{proposition}
Let $\lambda$ be the highest weight of an $\mathfrak{so}(2k)$ representation with $\lambda_k\geq 0$, then the dimension of the representation is connected to the dimension of representations with respect to $\mathfrak{gl}(2k)$ via the Littlewood-Richardson-coefficients $LR$ by 
\begin{equation}\label{Littlewood}
\dim V_\lambda^{\mathfrak{so}(2k)}
 =\sum_\mu\sum_{\alpha_1>\cdots >\alpha_s>0} (-1)^{\sum\alpha_j} LR_{\sigma(\alpha),\mu}^\lambda \dim V_\mu^{\mathfrak{gl}(2k)}\,.
\end{equation}
where the sum also contains $\alpha=(0)$.
\end{proposition}
Here $\sigma(\alpha)$ is the weight corresponding to the diagram which $i$th row has $\alpha_i+i$ boxes and its $i$th column has $\alpha_i+i-1$ boxes. The weights $\mu$ for which $LR_{\sigma(\alpha),\mu}^\lambda\neq 0$ correspond to traces of $\lambda$ associated to the symmetry of $\sigma(\alpha)$.

In the our case -- $\lambda_i\leq2$ -- only those coefficients are nonzero for which the diagram associated to $\sigma(\alpha)$ fits into the diagram associated to $\lambda$. The possible candidates are $\alpha=(0)$ with $\sigma(\alpha)=(0,\ldots,0)$ and $\alpha= (1) $ with $\sigma(\alpha)=(2,0,\ldots,0)$.  The weights $\mu$ such that $LR^{\lambda}_{\sigma(1),\mu}\neq0$ are associated to the  trace of $\lambda$ and if non vanishing the value of $LR$ is $1$. For $\lambda=(2^a1^b0^c)$ the  trace is given by $(2^{a-1}1^{b}0^{c+1})$.

If we define\footnote{This is the dimension of the $\mathfrak{gl}(2k)$ representation space with highest weight $(2^{k-p}1^{p-q}0^q)$.} 
\begin{equation}\label{fkpq}
\begin{split}
f(k,p,q)&:=
\frac{
\prod_{\alpha=1}^{k-q}(2k+1-\alpha) \prod_{\alpha=1}^{k-p}(2k+2-\alpha)
}{
\prod_{\alpha=1}^{k-p}(k-q+2-\alpha) \prod_{\alpha=1}^{p-q}\alpha\prod_{\alpha=1}^{k-p}\alpha
}\\
&=\frac{p-q+1}{2k+1}\binom{2k+1}{k-q+1}\binom{2k+1}{k-p}\,,
\end{split}
\end{equation} 
and $f(k,k+1,q)=0$ we have for all $k\geq p\geq q\geq 0$,
\begin{align*}
&\ f(k,p,q) -f(k,p+1,q+1) \\
=&\  \frac{(2k)!(2k+2)!(p+q+1)(p-q+1)}{(k+q+1)!(k+p+2)!(k-q+1)! (k-p)!}\\
=&\  \frac{(p+q+1)(p-q+1)}{(2k+1)(2k+2)}\binom{2k+2}{k-p}\binom{2k+2}{k-q+1}\,,
\end{align*}
so that the dimension of the $\mathfrak{so}(2k)$-representation $(2^{k-p}1^{p-q}0^q)^\pm$ is given by
\begin{align*}
\dim (2^{k-p}1^{p-q}0^q)  = 
		f(k,p,q) -f(k,p+1,q+1)  &\quad \text{if\ }\  k\geq p \geq q>0\,, \\
\intertext{and}
\dim (2^{k-p}1^{p})^\pm  = 
		\frac{1}{2}\big(f(k,p,0) -f(k,p+1,1)\big)  &\quad \text{if\ }\ k\geq p\geq q=0\,. 
\end{align*}
The proof of conjecture \ref{decspinorgrass} for $\ell=2$ is done by showing the identity
\begin{equation}
\begin{split}
&\ \sum_{0\leq p\leq k}\frac{(p+1)^4}{2}\binom{2k+2}{k-p}\binom{2k+2}{k+1}\\
& \quad  +\sum_{1\leq q\leq p\leq k}(p+q+1)^2(p-q+1)^2\binom{2k+2}{k-p}\binom{2k+2}{k-q+1} \\
=&\ (2k+2)(2k+1)2^{4k-1}\,.
\end{split}
\end{equation}
The left hand side of this can be written as 
\begin{align*}
&\ \sum_{1\leq  p\leq k+1}\frac{ p^4}{2}\binom{2k+2}{k+1- p}\binom{2k+2}{k+1}\\
&\quad +\sum_{1\leq q < p\leq k+1}( p-q)^2( p+q)^2\binom{2k+2}{k+1- p}\binom{2k+2}{k+1-q}\,.\\
\intertext{The terms of the second sum are symmetric with respect to $ p$ and $q$ and vanish for $ p=q$ so that we may sum over the whole square  instead of one triangle. So we may further write}
=&\ \frac{1}{2}\sum_{1\leq  p\leq k+1}  p^4\binom{2k+2}{k+1- p}\binom{2k+2}{k+1}\\
&\quad +\frac{1}{2}\sum_{1\leq q, p\leq k+1}( p- q)^2( p+q)^2\binom{2k+2}{k+1- p}\binom{2k+2}{k+1-q}\,.
\end{align*}
Therefore the following lemma -- which we will prove in appendix \ref{applemma} -- yields the subsequent theorem \ref{decspinorgrassl2}.
\begin{lemma}\label{lemma4}
\begin{equation}\label{lemmaeq}
\begin{split}
&\  
\sum_{i=0}^{n}\sum_{j=0}^{n}(i^2-j^2)^2\binom{2n}{n-i}\binom{2n}{n-j} 
 - \binom{2n}{n}\sum_{j=0}^n j^4\binom{2n}{n-j}\\ 
=&\ 2^{4n-3} n(2n-1)
\end{split}
\end{equation}
\end{lemma}
This proves theorem \ref{decspinorgrassl2}\,.

\subsection{The odd dimensional case }

We consider the odd dimensional spin Grassmannian $ G_{2k+1,2\ell+1}=SO(2k+2\ell+2)/SO(2k+1)\times SO(2\ell+1)$. The construction is almost the same as in the even dimensional case so that we will be more brief here. For the construction we take over the notation from section \ref{subsectionX} and add two more one dimensional spaces $V_{k+1}:=span\{v_{k+1}\},W_{\ell+1}:=span\{w_{\ell+1}\}$. We define in addition 
\begin{align*}
E_{i,\ell+1}&:=V_i\otimes W_{\ell+1}=span\{e_+^{i,\ell+1}, e_-^{i,\ell+1}\}
		&&\text{with } e_{+/-}^{i,\ell+1}:=v^i_{1/2}\otimes w_{\ell+1}\,,  \\
F_{k+1,i}&:=V_{k+1}\otimes W_i=span\{f_+^{k+1,i}, f_-^{k+1,i}\}
		&&\text{with } f_{+/-}^{k+1,i}:=v_{k+1}\otimes w^i_{2/1}\,,\\
G&:=span\{v_{k+1}\otimes w_{\ell+1}\}\,.&
\end{align*}
Then $\bigoplus E_{ij}\oplus \bigoplus F_{ij} \oplus G $ yields the corresponding decomposition of $\r^{(2k+1)(2\ell+1)}$ such that the Cartan basis of $so(4k\ell+2\ell+2k+1)$ is given by $\{N^e_{ij}\},N^f_{i'j'}\}$ with
\begin{align*}
N_{ij}^e(e_\pm^{i''j''})&=\pm\delta^{i''}_{i} \delta^{j''}_{j} e_\mp^{ij}		&&\text{for } 1\leq i\leq k, 1\leq j\leq \ell+1\,,\\
N_{i'j'}^f(f_\pm^{i''j''})&=\pm\delta^{i''}_{i'} \delta^{j''}_{j'} f_\mp^{i'j'}	&&\text{for } 1\leq i'\leq k+1, 1\leq j'\leq \ell\,.
\end{align*}
Then we get 
\begin{align*}
K_{i}&=\sum_{j=1}^\ell(N_{ij}^e-N_{ij}^f ) +N_{i,\ell+1}^e\,, \\
L_{j}&= -\sum_{i=1}^\ell(N_{ij}^e+N_{ij}^f ) -N_{k+1,j}^f \,.
\end{align*}
The corresponding matrix $A$ which is of size $(k+\ell)\times (2k\ell+k+\ell)$ has rows
\begin{align*}
A_i&= 
\big(
\overbrace{0,\ldots, 0}^{(i-1)(\ell+1)},
\overbrace{1\ldots,1}^{\ell+1},
\overbrace{0,\ldots,0}^{(k-i)(\ell+1)}, 
\overbrace{0,\ldots, 0}^{(i-1)\ell},
\overbrace{-1\ldots,-1}^{\ell},
\overbrace{0,\ldots,0}^{(k-i)\ell},
\overbrace{0,\ldots,0}^{\ell}
\big)\,, \\
A_{k+j}&= 
\big(
\underbrace{0,\ldots,0,\overset{\overset{j}{\downarrow}}{-1},0,\ldots0}_{\ell+1},
\ldots,
\underbrace{0,\ldots,0,\hspace{-1.5em}\overset{\overset{(k-1)(\ell+1)+j}{\downarrow}}{-1}\hspace{-1.5em},0,\ldots0}_{\ell+1},
	\\&\qquad\qquad \qquad 
\underbrace{0,\ldots,0,\hspace{-0.9em}\overset{\overset{k(\ell+1)+j}{\downarrow}}{-1}\hspace{-0.9em},0,\ldots,0}_{\ell},
\ldots,
\underbrace{0,\ldots,0,\hspace{-1.4em}\overset{\overset{k(\ell+1)+k\ell+j}{\downarrow}}{-1}\hspace{-1.4em},0,\ldots,0}_{\ell}
\big) \,,
\end{align*}
for $1\leq i \leq k$ and $ 1\leq j\leq \ell$. For example,
\begin{gather*}
A^{k=\ell=1}=
\begin{pmatrix}
	1&1&-1&0\\
	-1&0&-1&-1
\end{pmatrix}
\,,\quad
A^{k=2,\ell=1}=
\begin{pmatrix}
	1&1&0&0&-1&0&0\\
	0&0&1&1&0&-1&0\\
	-1&0&-1&0&-1&-1&-1
\end{pmatrix}\,,
\\
A^{k=\ell=2} =
\left(\begin{array}{cccccccccccc}
	1&1&1&0&0&0&-1&-1&0&0&0&0\\
	0&0&0&1&1&1&0&0&-1&-1&0&0\\
	-1&0&0&-1&0&0&-1&0&-1&0&-1&0\\
	0&-1&0&0&-1&0&0&-1&0&-1&0&-1
\end{array}\right)\,.
\end{gather*}
In contrast to the case of even dimensional spin Grassmannians  we have the following facts on the images.
\begin{remark}
The image of the set $\{\tfrac{1}{2}(\pm1,\ldots\pm1)^T\}$ by the map $A$ consist of weights which entries are contained in $\z+\frac{1}{2}$.
\end{remark}

\begin{conjecture}\label{decspinorgrassodd}
The spinor bundle of the odd dimensional spin Grassmannian decomposes into a sum of $\binom{k+\ell}{\ell}$ summands which are irreducible with respect to $\mathfrak{so}(2k+1)\oplus\mathfrak{so}(2\ell+1)$ and which are associated to the weights
\begin{equation}\label{decodd}
(\lambda_1+\tfrac{1}{2},\ldots,\lambda_k+\tfrac{1}{2} | \mu_1+\tfrac{1}{2}, \ldots \mu_\ell+\tfrac{1}{2})\,,
\end{equation}
with 
$0\leq \lambda_1\leq \ldots \leq \lambda_k\leq\ell$, $0\leq \mu_1\leq \ldots\leq \mu_\ell\leq k$\,,
and 
\[
(\lambda_1,\ldots,\lambda_k)^{c(k,\ell)}=(\mu_1,\ldots,\mu_\ell),\text{ or } (\mu_1,\ldots,\mu_\ell)^{c(\ell,k)}=(\lambda_1,\ldots,\lambda_k)\,.
\]
\end{conjecture}
\begin{theorem}\label{decspinorgrassoddpart}
Conjecture \ref{decspinorgrassodd} is true for $\ell\leq 2$.
\end{theorem}

That all the mentioned weights appear as image can be seen as in the even dimensional case by constructing an explicit preimage. The dimension argument uses the following observation (see, e.g., \cite{SamraKing}).
\begin{proposition}
Let $\vec{e}:=(1,\ldots,1)$. Then the dimension of the $\mathfrak{so}(2k+1)$-representation $V^{\mathfrak{so}(2k+1)}_{\lambda+\tfrac{1}{2}\vec{e}}$ associated to the weight $\lambda+\tfrac{1}{2}\vec{e}$ with $\lambda$ integer valued is given by  
\[
\dim V^{\mathfrak{so}(2k+1)}_{\lambda+\tfrac{1}{2}\vec{e}}  =  2^k \dim V^{\mathfrak{sp}(2k)}_\lambda \,,
\]
where $V^{\mathfrak{sp}(2k)}_\lambda$ is the $\mathfrak{sp}(2k)$-representation associated to the weight $\lambda$.
\end{proposition}
So we end up with an equation to verify which is similar to (\ref{master})
\begin{equation}\label{master2}
\sum_{\lambda\in(\ell^k)}\dim V_\lambda^{\mathfrak{sp}(2k)}\cdot\dim V_{\lambda^c}^{\mathfrak{sp}(2\ell)}=2^{2k\ell}\,.
\end{equation}

\subsection{The case $\ell\leq 2$}
In  the odd dimensional case beside our strong belief that the decomposition is true in general we have the following partial result.

\begin{proof}
\underline{$\ell=1$:\ }\ The representations which play a role are $(\lambda)$ with $1\leq\lambda\leq k$ and $\lambda^c= (1^{k-\lambda})$ with dimensions
\begin{equation*}
\dim V^{\mathfrak{sp}(2)}_\lambda=\lambda+1,\quad
\dim V^{\mathfrak{sp}(2k)}_{(1^{m})}= \frac{k+1-m}{k+1} \binom{2k+2}{m}\,.
\end{equation*}
This yields
\begin{align*}
\sum_{\lambda=0}^k \dim V^{\mathfrak{so}(2k+1)}_{(1^{k-\lambda})+\tfrac{1}{2}\vec{e}} 
		    \dim V^{\mathfrak{so}(3)}_{(\lambda)+\tfrac{1}{2}\vec{e}}
&= 2^{k+1} \sum_{\lambda=0}^k \dim V^{\mathfrak{sp}(2k)}_{(1^{k-\lambda})}
			         \dim V^{\mathfrak{sp}(2)}_{(\lambda)}\\
&= 2^{k+1} \sum_{\lambda=0}^k  \frac{(\lambda+1)^2}{k+1}  \binom{2k+2}{k-\lambda} \\
&= 2^{k+1} \frac{1}{k+1} B(k+1,2) \\
&=2^{3k+1}\,.
\end{align*}
\underline{$\ell=2$:\ }\ A formula similar to (\ref{Littlewood}) holds in the symplectic case:\footnote{We recall that $\lambda'$ is the diagram transpose to $\lambda$.}
\begin{equation}\label{Littlewood2}
\dim V_\lambda^{\mathfrak{sp}(2k)}
 =\sum_\mu\sum_{\alpha_1>\cdots >\alpha_s>0} (-1)^{\sum\alpha_j} LR_{\sigma(\alpha),\mu'}^{\lambda'} \dim V_\mu^{\mathfrak{gl}(2k)}\,.
\end{equation}
We list the relevant partitions $\mu$ which enter into the dimension formula in the symplectic case when we start with  $(2^{k-p},1^{p-q},0^q)$.
\begin{equation*}
\begin{array}{l|l|l|l}
\alpha	& \sigma(\alpha) 	& \mu 				&\text{condition}	\\\hline\hline
(1)	& \yc\tiny\yng(2) 	&(2^{k-p},1^{p-q-2},0^{q+2})	& p-q\geq2	\\
	&		&(2^{k-(p+1)},1^{p-q},0^{q+1})	& p-q\geq1	\\
	&		&(2^{k-(p+2)},1^{p-q+2},0^q)	& \text{none}		\\\hline
(2)	&\yc\tiny \yng(3,1)	& (2^{k-(p+1)},1^{p-q-2},0^{q+3})	& p-q\geq2	\\
	&		& (2^{k-(p+2)},1^{p-q},0^{q+2})	& p-q\geq1	\\
	&		& (2^{k-(p+3)},1^{p-q+2},0^{q+1})	& \text{none}		\\\hline
(2,1)	&\yc\tiny \yng(3,3)	& (2^{k-(p+3)},1^{p-q},0^{q+3})	& \text{none}	 	
\end{array}
\end{equation*}
We recall the dimension of the $\mathfrak{gl}(2k)$-representation $(2^{k-p},1^{p-q},0^q)$
\begin{align*}
f(k,p,q)&=\frac{p-q+1}{2k+1}\binom{2k+1}{k-q+1}\binom{2k+1}{k-p}\,,
\end{align*} 
see (\ref{fkpq}). 
Therefore the dimension of the $\mathfrak{sp}(2k)$-representation associated to the weight $(2^{k-p},1^{p-q},0^q)$ is 
\begin{align*}
\dim V^{\mathfrak{sp}(2k)}_{(2^{k-p},1^{p-q},0^q)}
&=  	f(k,p,q) - f(k,p+3,q+3) \\
&\quad 	+ \sum_{b=\max\{3-(p-q),1\}}^3  f(k,p+b,q+4-b) \\
&\quad	- \sum_{b=\max\{2-(p-q),0\}}^2  f(k,p+b,q+2-b) \,.
\end{align*}
Furthermore we have
\begin{align*}
\dim V_{(p,q)}^{\mathfrak{sp}(4)} &= \dim  V_{(p,q)}^{\mathfrak{gl}(4)}-\dim  V_{(p-1,q-1)}^{\mathfrak{gl}(4)}\\
	&= \frac{1}{6} (p-q+1)( p+q+3) (p+2)(q+1)\,.
\end{align*}
So we have to show
\begin{equation}\label{oddlequal2}
\sum_{0\leq q\leq  p\leq k} \dim V^{\mathfrak{sp}(2k)}_{(2^{k-p},1^{p-q},0^q)}\dim V_{(p,q)}^{\mathfrak{sp}(4)} =2^{4k}\,.
\end{equation}
This is done by expanding the left hand side which -- after a  lengthy calculation performed in appendix \ref{appenB} -- turns  it into a sum of products of at most two of the terms $B(k+1,m)$ for $m\leq 5$, see (\ref{Btilde})  in appendix \ref{applemma}. 
\end{proof}

\section{On the spectrum and the eigenspaces of the Dirac operator}\label{diracsection}
\subsection{General symmetric spaces}

The construction  described in section \ref{11} is used to calculate the spectrum of the eigenvalues of Dirac operators. This has been performed in \cite{Partha} in general and we give a short review. The $L^2$-sections of the spinor bundle on $G/K$  are identified with the $K$-equivariant maps from $G$ to the spinor module $\Delta$. Due to Frobenius reciprocity this can be further identified with $Hom_K(\c G,\Delta)=\bigoplus_{\lambda\in Irrep(G)}V_\lambda\otimes Hom_K(V_\lambda,\Delta)$ (see \cite{FultonHarris}). The square of the Dirac operator acts via $\dirac^2=C_\lambda +\frac{s}{8}$ where the Casimir $C_\lambda=c_\lambda\mathbbm{1}$ acts  proportionally to the identity due to the irreducibility of the representations. The factor  is $c_\lambda=b(\lambda+2\alpha_\mathfrak{g},\lambda)=\|\lambda+\alpha_\mathfrak{g}\|^2-\|\alpha_\mathfrak{g}\|^2$, where $\alpha_\mathfrak{g}$ is the weight (\ref{halfsums}) and $b$ is the metric induced by the Killing form of $\mathfrak{g}$.
Using this the spectrum of the square of the Dirac operator is given by
\begin{equation}
spec({\dirac}^2) = \big\{ c_\lambda + \tfrac{s}{8}\,\big|\, \lambda\in\mathcal{V}(G/K) \big\}\,,
\end{equation}
where the condition on the weight of the used  $G$-representations is   
\begin{equation}\label{eigen}
\mathcal{V}(G/K):=
\left\{ \lambda \,\left|\, 
	\text{\begin{tabular}{p{7cm}}
	$\lambda$ is highest weight of a $G$-irrep.\  s.t.\ one summand in its dec.\ w.r.t.\ $K$ is contained in the spinor dec.\ (\ref{spinordec}).
	\end{tabular}}
\right.\right\}\,.
\end{equation} 
In general the eigenvalue is degenerated in the sense that there exist $\{\lambda_j\}_{j=1,\ldots,N}$ such that $c_{\lambda_1}=\cdots= c_{\lambda_N}$.

The described construction makes use of branching rules for Lie algebras which usually are very hard to find. 
Nevertheless the theoretic basics provided so far can be used to calculate explicit examples as well as to formulate further general statements on the eigenvalue of the Dirac operator with the smallest absolute value (see, for example,  \cite{GoetteSemmelmann}, \cite{Seeger1}, \cite{Milhorat1}, \cite{Milhorat3}, \cite{CampoPedon} or \cite{Milhorat2}). 

In the following we will make some comments on the spectrum and the eigen\-spaces of the Dirac operator on the Grassmann manifold. 

\subsection{The even dimensional Grassmannians}

Given the decomposition of the spinor representation of the Grassmannian $G_{2k,2\ell}=SO(2k+2\ell))/(SO(2k)\times SO(2\ell))$ with respect to 
$\mathfrak{k}=\mathfrak{so}(2k)\oplus\mathfrak{so}(2\ell)$ as in conjecture \ref{decspinorgrass}.
Let $\mathfrak{g}=\mathfrak{so}(2(k+\ell))$ with $k\geq\ell\geq 2$ 
and  consider $\alpha_\mathfrak{k}$ and $\alpha_\mathfrak{g}$ as in (\ref{halfsums}). 
We have
\begin{align}
\alpha_\mathfrak{g} & 	
		      	=	\sum_{i=1}^{k+\ell}(k+\ell-i)e_i\, ,  \\
\alpha_\mathfrak{k}& 	
			=	\sum_{i=1}^k (k-i)e_i +\sum_{i=1}^\ell (\ell-i)e_{k+i}\,.
\end{align}
Let $\Pi$ be the set of highest weights corresponding to the decomposition of the spinor representation (\ref{spinordec}). 
Furthermore let $\Psi:=\big\{\gamma\in\Phi_\mathfrak{g}^+\, |\, \langle\gamma,\alpha_\mathfrak{k}\rangle<0 \big\}$

Due to \cite{Milhorat2} and \cite{Milhorat3} we are able to calculate at least the square of the smallest eigenvalue  $\epsilon_0$ of the Dirac operator as
\begin{align}
\epsilon_0^2  &= 2\min_{\beta\in\Pi}\|\beta\|^2 +\frac{k\ell}{2}\label{Milhoratform1}   \\
	      &= 2\|\alpha_\mathfrak{g}-\alpha_\mathfrak{k}\|^2
		+4\sum_{\gamma\in\Psi}\langle\gamma,\alpha_\mathfrak{k}\rangle +\frac{k\ell}{2} \label{Milhoratform2}  \,.
\end{align} 
Here $\langle\cdot,\cdot\rangle$ is the metric $b_{ij}=\frac{1}{4(k+\ell-1)}\delta_{ij}$ induced on the dual space by the Killing form (see \cite{FultonHarris}).

We observe that $\langle e_i+e_j,\alpha_\mathfrak{k}\rangle>0$ for all $i,j$ and $\langle e_i-e_j,\alpha_\mathfrak{k}\rangle<0$ only if $1\leq i\leq k$, $k+1 \leq j\leq k+\ell$.  More precisely we have
\[
\langle \alpha_\mathfrak{k},e_i-e_{k+j}\rangle_{\text{Eucl.}} = k-i-\ell+j\,,
\]
such that 
\begin{equation}
\Psi=\big\{e_i-e_{k+j} \big| 1\leq i\leq k, 1\leq j\leq\ell, i>k-\ell+j    \big\}\,.
\end{equation}
We have $||\alpha_\mathfrak{g}-\alpha_\mathfrak{k}||_{\text{Eucl.}}^2=||\ell\sum_{i=1}^ke_i||_{\text{Eucl.}}^2=k\ell^2$, so that we get 
\begin{align*}
4(k+\ell-1) (\epsilon_0^2 -\frac{k\ell}{2}) 	
 &= 	 2k\ell^2 + 4 \sum_{i=k-\ell+2}^{k}\sum_{j=1}^{i-1-(k-\ell)}(j - i+(k-\ell)) \\
 &= 	 \frac{2}{3}\big(3k\ell^2-(\ell^2-1)\ell\big)  \,.
\end{align*}

\begin{example}\label{ell2}
Suppose $\ell=2$. Writing $k=\frac{m}{2}$ we recover the result in \cite{Milhorat3}.  

In particular $k=4,\ell=2$ yields $\epsilon_0^2-\frac{k\ell}{2} = \frac{2\cdot14}{4(k+\ell-1)}$ where $14=\|(2220|11)\|_{\text{Eucl.}}^2=\|(2210|21)||_{\text{Eucl.}}^2=\|(2221|10)\|_{\text{Eucl.}}^2 =\|(2211|20)\|_{\text{Eucl.}}^2 $ is the minimum of the Euclidean norms of the weights from table \ref{table2}.

With $k=\ell=2$ we get $\epsilon_0^2-\frac{k\ell}{2}=\frac{2\cdot 6}{4(k+\ell-1)}$. Here  $6$ is the minimum of the Euclidean lengths of the weights corresponding to (\ref{spinor22}). There are four of them, e.g., $6=\|(11|20)\|^2_{\text{Eucl.}}$.
\end{example}

\begin{example}
$k=\ell=3$ yields $\epsilon_0^2-\frac{k\ell}{2}=\frac{2\cdot 19}{4(k+\ell-1)}$, where $19$ is the minimum of the Euclidean lengths obtained by eight weights  from table \ref{table3}, e.g., $19=\|(220|311)\|^2_{\text{Eucl.}}$.
\end{example}

\begin{proposition}
Let $k\geq\ell\geq2$. 
The square of the smallest eigenvalue of the Dirac operator on the Grassmannian $SO(2(k+\ell))/(SO(2k)\times SO(2\ell))$ is given by
\begin{equation}
  \epsilon_0^2=\frac{3k\ell^2-(\ell^2-1)\ell}{6(k+\ell-1)}+\frac{k\ell}{2}  	\label{smallest2}
\end{equation}
\end{proposition}

The examples above (in particular, example $\ref{ell2}$) can be generalized. Therefore assume $\ell\leq k$. The square of the Euclidean norm of a weight which appears in conjecture \ref{decspinorgrass} is
\begin{align*}
&\ \sum_{j=1}^{\ell}\lambda_j^2 +\ell^2(k-\lambda_1) +\sum_{j=1}^{\ell-1}(\ell-j)^2(\lambda_j-\lambda_{j+1})  \\
=\ & k\ell^2 +   \sum_{j=1}^{\ell}( (\lambda_j-\ell+j)-\tfrac{1}{2} )^2 -\frac{\ell}{4} - \frac{\ell(\ell^2-1)}{3}\,.
\end{align*}
This yields the following conjecture and proposition.
\begin{conjecture}\label{repsmallgencon}
Suppose $\ell\leq k$. There are $2^\ell$  weights which appear in the decomposition (\ref{dec}) of the spinor bundle of $SO(2k+2\ell)/SO(2k)\times SO(2\ell)$ and which are associated to the smallest eigenvalue of the square of the Dirac operator. They  are given by
\begin{align*}
(\ell^{k-\lambda_1},(\ell-1)^{\lambda_1-\lambda_2},\ldots, 1^{\lambda_{\ell-1}-\lambda_\ell},0^{\lambda_\ell}| \lambda_1,\ldots,\lambda_\ell)
	&&\text{with } \lambda_j\in\{\ell-j+1,\ell-j\}\,.
\end{align*}
\end{conjecture}
\begin{proposition}\label{repsmall2con}
The weights which appear in the decomposition of the spinor bundle of $SO(2k+4)/SO(2k)\times SO(4)$ and which are associated to the smallest eigenvalue of the square of the Dirac operator are given by
\begin{equation}\label{repsmall}
(2^{k-1}1|10),\quad (2^{k-1}0|11),\quad (2^{k-2}1^2|20),\quad (2^{k-2}10|21)\,.
\end{equation}
\end{proposition}

\begin{remark}
Both formulas (\ref{Milhoratform1}) and (\ref{Milhoratform2}) -- of which the first  uses the decomposition of the spinor bundle of the $G_{2k,2\ell}$ -- give the same value. This is once more a justification of the correctness of conjecture \ref{decspinorgrass}.
\end{remark}

We recall the following relations between $\mathfrak{so}(2k+2\ell)$-representations and $\mathfrak{so}(2k)$- as well as $\mathfrak{so}(2\ell)$-representations which are given by  the character formulas found in \cite{KoikeTerada4} and \cite{Koike1}
\begin{equation}\label{koik}
\begin{split}
\chi^{\mathfrak{so}(2k+2\ell)}_\mu \chi^{\mathfrak{so}(2k+2\ell)}_\nu 
	&=\sum 	LR_{\tau,\mu'}^\mu LR_{\tau,\nu'}^\nu LR_{\mu',\nu'}^{\lambda}
			\chi^{\mathfrak{so}(2k+2\ell)}_\lambda \,,\\
\chi^{\mathfrak{so}(2k+2\ell)}_\lambda 
	&= \sum 
		LR_{2\kappa,\lambda'}^{\lambda}LR_{\mu,\nu}^{\lambda'}
			\chi^{\mathfrak{so}(2k)}_\mu \chi^{\mathfrak{so}(2\ell)}_\nu\,.
\end{split}
\end{equation}
The first formula describes how the tensor product of two representations decompose, in particular only summands appear with 
$\sum\lambda_i\leq\sum\mu_i+\sum\nu_i$. 
In the second formula, the branching rule, only summands appear for which $\mu$ and $\nu$ are contained  in $\lambda$, i.e., $\sum\lambda_i\geq\sum\mu_i+\sum\nu_i$. 

The eigenspaces of the Dirac operator on the Grassmann manifold admit a decomposition with respect to the $\mathfrak{so}(2k+2\ell)$-representations -- see (\ref{eigen}). The ordering $\prec$ of  weights  (see, for example, \cite{Humphreys})  induce an ordering of the associated subspaces of the eigenspaces. From  formulas (\ref{koik}) we get the following remark.
\begin{remark}
Let $\lambda$ and $\mu$ be representations of $\mathfrak{so}(2k+2\ell)$, $\mathfrak{so}(2k)$, respectively, which associated Young diagrams we denote by the same symbol. Furthermore $\mu$ shall be contained in $(\ell^k)$ such that  $\mu^{c(k,\ell)}$ is a representation of $\mathfrak{so}(2\ell)$.  
Then $\mu\otimes\mu^c$ is contained in the branching of $\lambda$ if and only if there exists a $\mu^c$-expansion $\lambda'$ of the diagram $\mu$ such that $\lambda$ is a $2\kappa$-expansion of $\lambda'$. This yields a combinatorial  method to enumerate the eigenvalues of the square of the Dirac operator of $G_{2k,2\ell}$. 
1. Let $N(k,\ell)$ be the number of different  $\mu^c$-expansions of $\mu$ for $\mu\subset(\ell^k)$. 2. For the expansions $\{\lambda_i'\}_{i=1,\ldots,N(k,\ell)}$ consider the different $2\kappa$-expansions for all diagrams $\kappa$. 3. From the resulting diagrams $\lambda$ calculate $c_\lambda$.
Even in the case $G_{2k,4}$ this method seems to be very complex and one may suggest an analog approach as in \cite{Milhorat1} or \cite{Tsukamoto} -- in particular, because the ranks of the algebras are similar as in the first citation. But this fails due to the more complicated determinants which enter into the calculations so that the combinatorial ansatz is more applicable.
\end{remark}
\begin{example}\label{exampleexpansion}
For $k=\ell=2$ we have $\lambda'\in \Big\{{\tiny \yc\yng(2,2),\yng(3,1),\yng(2,1,1)}\Big\}$ such that the first $2\kappa$-expansions which branching admits a summand from (\ref{spinor22}) are from the following list.
\[{\setlength{\arraycolsep}{3pt}
\begin{array}{c|ccc|cccccccc}
2\kappa	&\multicolumn{3}{|c|}{\cdot} &\multicolumn{8}{|c}{\tiny\yng(2)}\\\hline
\lambda	& \tiny \yng(2,1,1)  &\tiny \yng(2,2) &\tiny\yng(3,1)
&\tiny \yng(2,2,1,1)&\tiny \yng(3,1,1,1)&\tiny \yng(2,2,2)&\tiny \yng(3,2,1)&\tiny \yng(4,1,1)&\tiny \yng(3,3)&\tiny \yng(4,2)&\tiny \yng(5,1) \\\hline
c_\lambda&24&28&32&32&36&36&42&48&48&52&60
\end{array}}
\]
In particular, the weight with minimal $c_\lambda$  must obey $\lambda_1\leq2$ and therefore the minimal eigenvalue is nondegenerated, i.e., the $\mathfrak{so}(8)$-representation $\lambda=(2,1,1,0)$ is the unique contribution to the eigenspace associated to the minimal eigenvalue.
\end{example}

We turn back to the smallest eigenvalue and notice that a contribution to the eigenspace of the Dirac operator to the smallest eigenvalue must obey
\[
\begin{split}
\|\lambda+\alpha_{\mathfrak{so}(2k+2\ell)}\|^2_{\text{Eucl}}- \|\alpha_{\mathfrak{so}(2k+2\ell)}\|^2_{\text{Eucl}}
&= 3 k\ell^2+k^2\ell-k\ell - \tfrac{2}{3}\ell(\ell^2-1)\\
&\overset{\ell=2}{=} 2k^2 + 10k -4\,.
\end{split}
\]
\begin{conjecture}\label{smallcontr}
The smallest contribution to the eigenspace  of the smallest eigenvalue of the square of the Dirac operator on $G_{2k,2\ell}$ is given by the weight  
\begin{equation}
\lambda^0=(\ell^{k-\ell+1},(\ell-1)^2,(\ell-2)^2,\ldots,1^2,0)\,.
\end{equation}
\end{conjecture}
\begin{proposition}\label{smallcontrl2}
Conjecture \ref{smallcontr} is true for $\ell=2$.
\end{proposition}
\begin{proof}
Due to (\ref{koik}) the smallest contribution to the eigenspace comes from the biggest contribution in the tensor product. If we take the weights from conjecture \ref{repsmallgencon} the ordering $\prec$ of the weights then sorts out the smallest one, i.e., the one with $\lambda_1\leq \ell$. A straightforward calculation shows that $c_{\lambda^0}^{\text{Eucl}}=3 k\ell^2+k^2\ell-k\ell - \tfrac{2}{3}\ell(\ell^2-1)$. 
\end{proof}

\begin{example}
Consider the case $\ell=2$. The representations which we obtain by decomposing the tensor products $(\mu|\nu)$ from (\ref{repsmall}) are given in the next table where we also added $c_\lambda$.
\begin{equation*}
\begin{array}{l|l}
\lambda			& c_\lambda^{\text{Eucl}}  \\\hline
(4,3,2^{k-4},1,0^3) 	&2k^2+16k	\\
(4,2^{k-2},0^3)		&2k^2+14k+4	\\
(4,2^{k-3},1^2,0^2)	&2k^2+14k	\\
(3^2,2^{k-3},0^3)		&2k^2+14k	\\
(3,2^{k-2},1,0^2)		&2k^2+12k	\\
(3,2^{k-3}1^3,0)		&2k^2+12k-6	\\
(2^k,0^2)		&2k^2+10k	\\
(2^{k-1}1^2,0)		&2k^2+10k-4	
\end{array}
\end{equation*}
We see, that in this case $\lambda^0$ is unique, if we restrict to the  weights coming from tensor products. 
\end{example}

\section{Concluding remarks}\label{outlook}

{\em Remark on the general proof of (\ref{master}) and (\ref{master2}).\ }\ 
In principle, all we need  for the proof of conjectures \ref{decspinorgrass} and \ref{decspinorgrassodd} for fixed $\ell\geq 3$ is appendix \ref{applemma} and the dimension formulas (\ref{Littlewood}) and (\ref{Littlewood2}). This worked very well in the cases $\ell\leq 2$.
In (\ref{master}) and (\ref{master2}) the dimension of $V_\lambda$ for $\lambda=(\lambda_1,\ldots,\lambda_\ell)$ is a polynomial depending on the $\ell$ entries. 
If we use  (\ref{lambdaC}) to describe $\lambda^{c(\ell,k)}$ and insert the dimension formulas (\ref{Littlewood}) or (\ref{Littlewood2}) as well as (\ref{fkpq})   into  (\ref{master}) or (\ref{master2}) we end up with sums of at most $\ell$-fold products of terms of the form $B(k+1,m)$ with $m$ increasing for $\ell$ increasing.
For large $\ell$ -- beside the calculation of $B(k+1,m)$ for large $m$ (which may be done by  computer using the generating function (\ref{generating})) --  the main difficulties arise when we try to describe uniformly the $\mu$ such that $LR_{\sigma(\alpha),\mu}^{\lambda^c}$ do not vanish. In particular, this problem limits the practical application of our ansatz for the proof. 

{\em Remark on the remaining case $\mathfrak{so}(2k)\oplus\mathfrak{so}(2\ell+1)\subset \mathfrak{so}(4k\ell+2k)$.\ }\ 
One thing which has been left out in section \ref{sectionX} is the decomposition of the spin representation of $SO(4k\ell+2k)$ with respect to the subgroup $SO(2k)\times SO(2\ell+1)$ where the embedding is such that the vector representation decomposes as $\mathbf{2k}\otimes(\mathbf{2\ell+1})$. We did not mention this because the symmetric space $G_{2k,2\ell+1}$ is not spin. Nevertheless on the level of representation the question of branching is interesting as well. We will only state the result which is obtained by the same method as used in the two spin cases. In particular, the cases $(2k,3)$, $(2k,5)$, and $(4,2\ell+1)$ can be proven in the same way and the remark extends to this case, too.

The spin representation $S=S^+\oplus S^-$ of $\mathfrak{so}(4k\ell+2k)$ branches with respect to $\mathfrak{so}(2k)\oplus\mathfrak{so}(2\ell+1)$ into the sum of the following representations:
\[
( \lambda+\tfrac{1}{2}\vec{e}{\,}^\pm  | \lambda^c) \ \text{ with }\ \lambda=(\lambda_1\ldots,\lambda_k) \subset(\ell^k)\,,
\]
where $\lambda+\tfrac{1}{2}\vec{e}{\,}^\pm=(\lambda_1+\tfrac{1}{2},\ldots,\lambda_{k-1}+\tfrac{1}{2},\pm(\lambda_k+\tfrac{1}{2}))$ belongs to the subrepresentation $S^\pm$.

{\bf Acknowledgments:\ }\ I would like to thank Lorenz Schwachh\"ofer for useful comments and Doron Zeilberger for suggesting a method to calculate (\ref{BNK}) for  $k$ even without  generating functions.

\begin{appendix}

\section{Useful binomial identities}\label{applemma}

We define sums 
\begin{equation}\label{BNK}
B(n,k):=\sum_{j=0}^nj^k\binom{2n}{n-j}\,,
\end{equation}
such that lemma \ref{lemma4} in these terms reads as
\begin{equation}
2B(n,4)B(n,0)-2B(n,2)^2-\binom{2n}{n}B(n,4)=2^{4n-3}n(2n-1)\,.
\tag{\ref{lemmaeq}}
\end{equation}
We provide the tools to calculate (\ref{BNK}) for all values of $n$ and $k$ although for the proof of (\ref{lemmaeq}) and (\ref{oddlequal2}) we only need them for values less than 6.
 
We consider polynomials
\begin{equation}\label{polypn}
P_n(x):=\sum_{i=0}^n\binom{2n}{n-i}x^i\,,
\end{equation}
which obey  
\begin{equation}\label{BP}
\begin{split}
P_n(1)		&=B(n,0)\,,\\
P_n'(1)		&=B(n,1)\,,\\
P_n''(1)		&=B(n,2)-B(n,1)\,,\\
P_n'''(1)		&=B(n,3)-3B(n,2)+2B(n,1)\,,\\
P_n''''(1)	&=B(n,4)-6B(n,3)+11B(n,2)-6B(n,1)\,,\\
P_n'''''(1)	&=B(n,5)-10B(n,4)+35B(n,3)-50B(n,2)+24B(n,1)\,,
\end{split}
\end{equation}
and 
\begin{align*}
P_{n+1}(x)& =\sum_{j=0}^{n+1}\binom{2n+2}{n+1-j}x^j \\
&= \sum_{j=0}^{n+1}\big( \binom{2n}{n-j+1}+2 \binom{2n}{n-j}+ \binom{2n}{n-j-1}  \big)x^j \\
&= x\sum_{j=0}^{n} \binom{2n}{n-j} x^{j} 
	+2 \sum_{j=0}^{n}\binom{2n}{n-j}x^j
	+ \frac{1}{x}\sum_{j=0}^{n}\binom{2n}{n-j}  x^{j}\\
& \quad +\binom{2n}{n+1}-\frac{1}{x}\binom{2n}{n} \\
&=\frac{(x+1)^2}{x}P_n(x) +\frac{x-1}{x}\binom{2n}{n}-\frac{1}{n+1}\binom{2n}{n} \,.
\end{align*}
\begin{proposition}
The generating function $Q$ of the polynomials $P_n$ cf.\ (\ref{polypn}) is given by 
\begin{equation}\label{generating}
Q(x,y)=\frac{y(x-1)}{(x-y(x+1)^2)\sqrt {1-4y}} 
	+\frac{x(\sqrt{1-4y} +1)}{2(x-y(x+1)^2)} \,.
\end{equation}

\end{proposition}
The function $Q$ obeys
\begin{equation}
\frac{\partial^{k+n}Q}{\partial x^k \partial y^n} (1,0)= n!P^{(k)}_n(1)\,,
\end{equation}
such that an explicit description for $Q$ allows us  to calculate all $B(n,k)$ which then proves lemma \ref{lemma4}.
\begin{proof}
We define
\begin{equation}\label{QP}
Q(x,y):=\sum_{n\geq 0} P_n(x)y^n
\end{equation}
and  use the above recursion formula for $P_n$ as well as 
\begin{align*}
(1-4y)^{-\frac{1}{2}}&=\sum_{m\geq0} \frac{\frac{1}{2}(\frac{1}{2}+1)\ldots(\frac{1}{2}+m-1)}{m!}(4y)^m
= \sum_{m\geq0} \binom{2m}{m}y^m\,, 
\end{align*}
and
$ \displaystyle \int (1-4y)^{-\frac{1}{2}} = -\frac{1}{2}(1-4y)^\frac{1}{2} $
which yields
\begin{align*}
Q(x,y) &=1 + \sum_{n\geq1} P_n(x)y^n
\ =\ 1 	+ y\sum_{n\geq0} P_{n+1}(x)y^{n} \\
&= 1	+\frac{y(x+1)^2}{x}\sum_{n\geq0} P_n(x) y^n
	+\frac{y(x-1)}{x}\sum_{n\geq0} \binom{2n}{n}y^n \\
&\quad 	-y \sum_{n\geq0} \frac{1}{n+1}\binom{2n}{n}y^n \\
&= 	\frac{y(x+1)^2}{x}Q(x,y) \\
&\quad	+1 +\frac{y(x-1)}{x}\sum_{n\geq0}\binom{2n}{n}y^n
	- \int_0^y \sum_{n\geq0} \binom{2n}{n}t^n  dt \\
&=  	\frac{y(x+1)^2}{x}Q(x,y) +1 
	+\frac{y(x-1)}{x\sqrt {1-4y}}
	+\frac{1}{2}\sqrt{1-4y} -\frac{1}{2} \,,
\end{align*}
from which we get (\ref{generating}).
\end{proof}
The derivatives of $Q$ with respect to $x$ at the point $(1,y)$ are 
\begin{align*}
Q(1,y) &= \frac{1}{2}(1-4y)^{-\frac{1}{2}}+ \frac{1}{2}(1-4y)^{-1}\,, \\
\frac{\partial Q}{\partial x}(1,y) &= y(1-4y)^{-\frac{3}{2}} \,,\\
\frac{\partial^2 Q}{\partial x^2}(1,y) &=-y(1-4y)^{-\frac{3}{2}}+y(1-4y)^{-2}\,,  \\
\frac{\partial^3 Q}{\partial x^3}(1,y) &=3y(1-4y)^{-\frac{3}{2}}+6y^2(1-4y)^{-\frac{5}{2}}-3y(1-4y)^{-2}\,,  \\
\frac{\partial^4 Q}{\partial x^4}(1,y) &= -3y(1-4y)^{-\frac{3}{2}}-9y(1-4y)^{-\frac{5}{2}}+9y(1-4y)^{-2} +3y(1-4y)^{-3}\,,\\
\frac{\partial^5 Q}{\partial x^5}(1,y) &= 120y^3(1-4y)^{-\frac{7}{2}}-60y(1-4y)^{-2}+60y(1-4y)^{-\frac{5}{2}}-120y^2(1-4y)^{-3}.
\end{align*}
All terms are of the form $y^m(1-4y)^{-\ell}$ with $m=0,1,2,3$ and the needed derivatives are  given by
\begin{align*}
\frac{d^n}{dy^n}(1-4y)^{-\ell}\Big|_{y=0}
	&=\begin{cases}
		\displaystyle 2^{2n}\frac{(\ell+n-1)!}{(\ell-1)!}\,,  & \ell\in \n \\[1.5ex]
		\displaystyle \frac{(2n+2\ell-1)!  (\ell-\frac{1}{2})!}{(n+\ell-\frac{1}{2})!(2\ell-1)!}\,, & \ell\in \n+\frac{1}{2}
	    \end{cases}\\
\intertext{as well as}
\frac{d^n}{dy^n}(y^k(1-4y)^{-\ell})\Big|_{y=0}&
=\begin{cases} 
	\displaystyle \binom{n}{k}k! \frac{d^{n-k}}{dy^{n-k}}(1-4y)^{-\ell}\Big|_{y=0}&\  k\leq n\\[1.5ex]
	0 & \text{ else}
\end{cases}
\end{align*}
Therefore the partial derivatives of $Q$ with respect to $y$  at $(1,0)$ are
\begin{align*}
\frac{\partial^n Q}{\partial y^n}(1,0) 
	&= 2^{2n-1}n!+2^{-1}\frac{(2n)!}{n!} \,,\\
\frac{\partial^{n+1} Q}{\partial y^n\partial x}(1,0) 
	&= 2^{-1}\frac{(2n)!}{n!}n \,,\\
\frac{\partial^{n+2} Q}{\partial y^n\partial x^2}(1,0) 
	&=  -2^{-1}\frac{(2n)!}{n!}n+2^{2n-2} n!n\,,\\
\frac{\partial^{n+3} Q}{\partial y^n\partial x^3}(1,0) 
	&= 2 ^{-1}\frac{(2n)!}{n!}n( n+2)  -3\cdot 2^{2n-2} n!n  \,,\\
\frac{\partial^{n+4} Q}{\partial y^n\partial x^4}(1,0) 
	&= -3 \frac{(2n)!}{n!}n( n+1)  +3\cdot 2^{2n-3}n!n(n+7)\,,\\
\frac{\partial^{n+5} Q}{\partial y^n\partial x^5}(1,0) 
	&=  \frac{(2n)!}{n!}n( n^2+17n+12)  -15\cdot 2^{2n-2}n! n(n+3) \,.	
\end{align*}
From this we get with (\ref{BP}) and  (\ref{QP}):
\begin{equation}\label{resBNK}\begin{aligned}
B(n,0)	&
		=2^{2n-1}+\frac{1}{2}\binom{2n}{n}\,, &\qquad
B(n,1)	&
		=\frac{1}{2}\binom{2n}{n}n\,,		\\
B(n,2)	&
		=2^{2n-2}n\,,			&\qquad
B(n,3)	&
		=\frac{1}{2}\binom{2n}{n}n^2\,,		\\
B(n,4)	&
		=2^{2n-3}n (3n-1) \,,		&\qquad
B(n,5)	&
		=\frac{1}{2}\binom{2n}{n}n^2(2n-1)\,,
\end{aligned}\end{equation}
which -- last but not least -- yields lemma \ref{lemma4}.

A variant  of (\ref{BNK}) is used in the proof of theorem \ref{decspinorgrassoddpart}. We define
\begin{equation}\label{Btilde}
\widetilde{B}(n,k):=\sum_{j=0}^{n}j^k\binom{2n-1}{n-j}\,.
\end{equation}
Then usual manipulation of the binomials yields the following lemma
\begin{lemma}
The sums $B(n,k)$ and $\widetilde{B}(n,k)$ are connected via
\begin{equation}\label{resBtilde}
2n\widetilde{B}(n,k)=n B(n,k)+B(n,k+1)\,.
\end{equation}
\end{lemma}

\section{Calculations for the proof of (\ref{oddlequal2}).}\label{appenB}

For the proof of (\ref{oddlequal2}) we expand the left hand side and notice that we can extend the sums for $b$, because the added summands cancel or vanish. This yields
\begin{align*}
  &6(2k+1) \sum_{0\leq q\leq p\leq k} \dim V^{\mathfrak{sp}(2k)}_{(2^{k-p},1^{p-q},0^q)}\dim V^{\mathfrak{sp}(4)}_{(p,q)} \displaybreak[2]\\
=&   \sum_{p=0}^k\sum_{q=0}^p
	(p-q+1 )^2( p+q+3)(p+2)(q+1)\binom{2k+1}{k-q+1}\binom{2k+1}{k-p} \displaybreak[2]\\
  &- \sum_{p=0}^k\sum_{q=0}^p
	(p-q+1 )^2( p+q+3)(p+2)(q+1)\binom{2k+1}{k-q-2}\binom{2k+1}{k-p-3}\displaybreak[2] \\
  &+ \sum_{p=0}^k\sum_{q=0}^p\sum_{b=1}^3
	(p-q+1 )(p-q-3+2b)( p+q+3)(p+2)(q+1)\cdot \\
  &\qquad	\cdot\binom{2k+1}{k-q+b-3}\binom{2k+1}{k-p-b}\displaybreak[2] \\
  &- \sum_{p=0}^k\sum_{q=0}^p\sum_{b=0}^2
	(p-q+1 )(p-q-1+2b)( p+q+3)(p+2)(q+1)\cdot \\
&\qquad 		\cdot\binom{2k+1}{k-q+b-1}\binom{2k+1}{k-p-b} \displaybreak[2] \\
=&   \sum_{p=0}^k\sum_{q=0}^p
	(p-q+1)^2( p+q+3)(p+2)(q+1)\binom{2k+1}{k-q+1}\binom{2k+1}{k-p}\displaybreak[2]  \\
  &- \sum_{p=0}^k\sum_{q=0}^p
	(p-q+1)^2( p+q+3)(p+2)(q+1)\binom{2k+1}{k-q-2}\binom{2k+1}{k-p-3} \displaybreak[2]\\
  &+ \sum_{p=0}^k\sum_{q=0}^p\sum_{b=0}^2
	(p-q+1)(p-q-1+2b)( p+q+3)(p+2)(q+1)\cdot \\
&\qquad 		\cdot \binom{2k+1}{k-q+b-2}\binom{2k+1}{k-p-b-1}\displaybreak[2] \\ 
  &- \sum_{p=0}^k\sum_{q=0}^p\sum_{b=0}^2
	(p-q+1)(p-q-1+2b)( p+q+3)(p+2)(q+1)\cdot \\
&\qquad 		\cdot \binom{2k+1}{k-q+b-1}\binom{2k+1}{k-p-b}\,.
\end{align*}
With
\[
\binom{n}{m}\binom{n}{\ell}-\binom{n}{m-1}\binom{n}{\ell-1}=\frac{n-m-\ell+1}{n+1}\binom{n+1}{m}\binom{n+1}{\ell}
\]
we get
\begin{align*}
=&   \sum_{p=0}^k\sum_{q=0}^p
	(p-q+1)^2( p+q+3)(p+2)(q+1)\binom{2k+1}{k-q+1}\binom{2k+1}{k-p} \displaybreak[2]\\
  &- \sum_{p=0}^k\sum_{q=0}^p
	(p-q+1)^2( p+q+3)(p+2)(q+1)\binom{2k+1}{k-q-2}\binom{2k+1}{k-p-3}\displaybreak[2] \\
  &- \sum_{p=0}^k\sum_{q=0}^p\sum_{b=0}^2
	\frac{(p-q+1)(p-q-1+2b)( p+q+3)^2(p+2)(q+1)}{2k+2}\cdot \\
&\qquad 		\cdot\binom{2k+2}{k-q+b-1}\binom{2k+2}{k-p-b}\displaybreak[2]\\
=&   \sum_{p=0}^{k+1}\sum_{q=0}^p
	(p-q)^2( p+q+2)(p+1)(q+1)\binom{2k+1}{k-q+1}\binom{2k+1}{k-p+1}\displaybreak[2] \\
  &- \sum_{p=0}^{k+1}\sum_{q=0}^p
	(p-q)^2( p+q+2)(p+1)(q+1)\binom{2k+1}{k-q-2}\binom{2k+1}{k-p-2}\displaybreak[2] \\
  &- \sum_{p=1}^{k+1}\sum_{q=1}^{p+2}
	\frac{(p-q+2)(p-q)( p+q)^2(p+1)(q-1)}{2k+2}\binom{2k+2}{k-q+1}\binom{2k+2}{k-p+1}\displaybreak[2]\\
  &- \sum_{p=2}^{k+1}\sum_{q=0}^{p}
	\frac{(p-q)^2( p+q)^2 pq}{2k+2}\binom{2k+2}{k-q+1}\binom{2k+2}{k-p+1}\displaybreak[2]\\
  &- \sum_{q=3}^{k+1}\sum_{p=0}^{q-2}
	\frac{(p-q+2)(p-q)(q+p)^2(q-1)(p+1)}{2k+2}\binom{2k+2}{k-q+1}\binom{2k+2}{k-p+1}\,.
\end{align*}
We use the symmetry of the summands with respect to $p$ and $q$ so that we arrive at
\begin{align*}
=&   \frac{1}{2}\sum_{p=0}^{k+1}\sum_{q=0}^{k+1}
	(p-q)^2( p+q+2)(p+1)(q+1)\binom{2k+1}{k-q+1}\binom{2k+1}{k-p+1} \\
  &- \frac{1}{2}\sum_{p=0}^{k+1}\sum_{q=0}^{k+1}
	(p-q)^2( p+q-4)(p-2)(q-2)\binom{2k+1}{k-q+1}\binom{2k+1}{k-p+1}\displaybreak[2] \\
  &- \sum_{p=0}^{k+1}\sum_{q=0}^{k+1}
	\frac{(p-q+2)(p-q)( p+q)^2(p+1)(q-1)}{2(k+1)}\binom{2k+2}{k-q+1}\binom{2k+2}{k-p+1}\displaybreak[2]\\
  &- \sum_{p=0}^{k+1}\sum_{q=0}^{k+1}
	\frac{(p-q)^2( p+q)^2 pq}{4(k+1)}\binom{2k+2}{k-q+1}\binom{2k+2}{k-p+1}\displaybreak[2]\\
  &- \sum_{p=0}^{k+1}
	\frac{p^3(p+1)(p+2)}{2(k+1)}\binom{2k+2}{k-p+1}\binom{2k+2}{k+1}\displaybreak[2]\\
  &-2 \sum_{p=0}^{k+1}
	p^2(p-4)(p-2)\binom{2k+1}{k+1}\binom{2k+1}{k-p+1}\displaybreak[2] \\  
  &- \sum_{p=0}^{k+1}
	(p-1)^2( p-3)(p-2)\binom{2k+1}{k}\binom{2k+1}{k-p+1}\displaybreak[2] \\
  &+6 \binom{2k+1}{k}\binom{2k+1}{k+1}\displaybreak[2] \\
=&   \frac{1}{2}\sum_{p=0}^{k+1}\sum_{q=0}^{k+1}
	  (p-q)^2\big(( p+q+2)(p+1)(q+1)-( p+q-4)(p-2)(q-2)\big)\cdot \\
 &\qquad  	\cdot\binom{2k+1}{k-q+1}\binom{2k+1}{k-p+1}\displaybreak[2] \\
 &- \frac{1}{4(k+1)}\sum_{p=0}^{k+1}\sum_{q=0}^{k+1}
	 (p+q)^2(p-q)\big( 2(p-q+2)(p+1)(q-1)+(p-q) pq\big)\cdot \\
 &\qquad 	\cdot\binom{2k+2}{k-q+1}\binom{2k+2}{k-p+1}\displaybreak[2]\\
 &- \frac{1}{2(k+1)}\binom{2k+2}{k+1}\sum_{p=0}^{k+1}
	p^3(p+1)(p+2)\binom{2k+2}{k-p+1}\displaybreak[2]\\
 &-\binom{2k+1}{k}\sum_{p=0}^{k+1}
	 (p-2)(2p^2(p-4)+(p-1)^2( p-3)\binom{2k+1}{k-p+1}\,.\displaybreak[2] \\
&+6\binom{2k+1}{k}^2
\end{align*}
We expand the polynomial coefficients and write the last sum in Terms of $\widetilde{B}$ and $B$:
\begin{align*}
=& 	     \widetilde{B}(k+1,4)\widetilde{B}(k+1,0)
	+6\widetilde{B}(k+1,3)\widetilde{B}(k+1,1)
  	-9 \widetilde{B}(k+1,2)^2\\
&	-9 \widetilde{B}(k+1,3)\widetilde{B}(k+1,0)
	+9 \widetilde{B}(k+1,2)\widetilde{B}(k+1,1)\\
&	+18 \widetilde{B}(k+1,2)\widetilde{B}(k+1,0)
	-18 \widetilde{B}(k+1,1)^2\\
&-\frac{1}{2(k+1)}\Big(
	  3   B(k+1,5)B(k+1,1)
	-6 B(k+1,4)B(k+1,0) 
	-3  B(k+1,3)^2\\
&	+6 B(k+1,2)^2  
+ \binom{2k+2}{k+1}\big( 
	     B(k+1,5)
	+3 B(k+1,4)
	+2 B(k+1,3) \big)\Big) \\
&-\binom{2k+1}{k}\Big(
	 2    \widetilde{B}(k+1,4)
	-12 \widetilde{B}(k+1,3)
	+16  \widetilde{B}(k+1,2) 
	+    \widetilde{B}(k+1,4)\\
&	-7 \widetilde{B}(k+1,3)
	+17\widetilde{B}(k+1,2)
	-17\widetilde{B}(k+1,1)
	+6  \widetilde{B}(k+1,0)  \Big) \\
& +6 \binom{2k+1}{k}\binom{2k+1}{k+1}\,.
\end{align*}
Inserting (\ref{resBtilde}) as well as (\ref{resBNK}) from appendix \ref{applemma} yields (\ref{oddlequal2}) which finishes the proof of theorem \ref{decspinorgrassoddpart}.

\end{appendix}

\end{document}